\documentclass[11pt]{amsart}

\usepackage[a4paper,margin=1.1in]{geometry}
\usepackage{amsmath,amsthm,amssymb,mathtools}
\usepackage{hyperref}
\usepackage{enumitem}
\usepackage [autostyle, english = american]{csquotes}
\MakeOuterQuote{"}

\newtheorem{thm}{Theorem}[section]
\newtheorem{lem}[thm]{Lemma}
\newtheorem{prp}[thm]{Proposition}
\newtheorem{cor}[thm]{Corollary}
\newtheorem{dfn}[thm]{Definition}
\newtheorem{rmk}[thm]{Remark}

\numberwithin{equation}{section}

\newcommand{\R}{\mathbb{R}}

\title{Obstructions to exact regularity of sublaplacians}
\author{Gian Maria Dall'Ara}
\thanks{Istituto Nazionale di Alta Matematica ``F. Severi" and Scuola Normale Superiore}

\date{\today}

\begin{document}
	
\maketitle

\begin{abstract}
Let $M$ be a closed manifold equipped with a collection of smooth real vector fields $X_1,\ldots, X_k$ and a smooth measure $\mu$. Let $P:=\sum_j X_j^\dagger X_j $ be the associated sublaplacian. Under the assumption that each pair of points of $M$ can be connected by a path obtained by joining integral curves of the vector fields, the equation \[
Pu=f
\] admits unique zero-average solutions $u$ in the natural energy space, for all zero average data $f\in L^2$. We say that exact regularity holds for the above equation if $u\in H^k(M)$ whenever $f\in H^k(M)$, where $H^k(M)$ is the $L^2$ based Sobolev space of any order $k\in \mathbb{N}$. We investigate how the presence of a characteristic submanifold, namely a submanifold tangent to all vector fields $X_j$, may cause a failure of exact regularity. We prove that this indeed happens for a class of "worm sublaplacians", which are real analogues of Kohn Laplacians on Diederich--Fornaess worm domains, of interest in several complex variables. Our main tool is a scaling lemma inspired by work of Barrett and Christ on the $\bar\partial$-Neumann problem. 
\end{abstract}

\tableofcontents

\section{Introduction} 

\subsection{The exact regularity problem for sublaplacians} This paper is a contribution to the global regularity theory of degenerate linear PDEs on manifolds. It investigates obstructions to "exact regularity" for general sublaplacians. We say that exact regularity holds for a PDE \[
Pu=f
\] on a manifold $M$ if we have the implication \[f\in H^k(M)\qquad \Longrightarrow\qquad u\in H^k(M)\] for all $k$, where $H^k(M)$ is the $L^2$ based Sobolev space of order $k$; in other words, exact regularity holds if the solutions have at least the same amount of ($L^2$ Sobolev) regularity as the data. In order to make the notion completely rigorous, one first needs to establish a basic existence and uniqueness theory for the equation of interest, specifying to which weak solutions exact regularity applies (technically, the Sobolev space $H^k(M)$ is not canonically defined in case the manifold is noncompact, but this will not be an issue for us, as we will confine our attention to the compact case). 

We stress the fact that exact regularity is a \emph{global} property: it only applies to global solutions $u$ for which the datum $f$ is globally regular, and may hold even in absence of any local regularity. By Sobolev embedding exact regularity implies global regularity, that is,  \begin{equation}\label{eq:global_reg} f\in C^\infty(M)\qquad \Longrightarrow\qquad u\in C^\infty(M), \end{equation}
while it has no bearing on local regularity properties like hypoellipticity (we refrained from calling the implication \eqref{eq:global_reg} "global hypoellipticity" as it may apply only to a specific class of solutions). \newline 

The linear PDEs considered in this paper are sublaplacians \[
P=\sum_j X_j^\dagger X_j, 
\]
where $\{X_1,\ldots, X_k\}$ is a collection of smooth real vector fields on a compact connected manifold $M$, and $\dagger$ denotes formal adjunction with respect to a background measure $\mu$ (always assumed to have a smooth positive density in local coordinates). 
It has been known since the landmark work \cite{Hormander1967} of Hörmander that, if the Lie algebra generated by $X_1, \ldots, X_k$ spans the tangent space of $M$ at a point $p$ (this is often called the "Hörmander iterated commutators" condition), then solutions of the equation $Pu=f$ gain a positive amount of regularity in the $L^2$ Sobolev scale, locally near the point $p$: if $f\in H^k$ in a neighborhood of $p$, then $u\in H^{k+\delta}$ on a possibly smaller neighborhood; here $\delta$ is positive and depends on the number of iterated commutators needed to span the tangent space. In particular, if the Hörmander iterated commutators condition holds at every point, then $P$ is certainly exactly regular. Thus, any further progress on exact regularity for sublaplacians requires an analysis of the "degeneration locus", namely the set of points where Hörmander condition fails. 

\subsection{Informal presentation of results} In a previous paper \cite{DallAra2024}, we formulated a geometric condition on the degeneration locus of a sublaplacian that ensures global regularity via a classical elliptic regularization technique of Kohn and Nirenberg (see also \cite{DallAra2024} for a discussion of antecedents of this result in the literature). In fact, the same argument shows that exact regularity holds for $P$ under the same condition (this is not explicitly stated in \cite{DallAra2024}, but follows from \cite[Appendix A]{DallAra2024} and \cite{KohnNirenberg1965}). 

Our geometric condition is roughly speaking the following (we refer the reader to \cite[Section 3]{DallAra2024} for details). Start with the degeneration locus $Z_0$; remove from it all points at which at least one of the vector fields $X_j$ is transversal to $Z_0$ and denote by $Z_1$ the remaining set; iterate, obtaining a descending chain of subsets $Z_0\supseteq Z_1\supseteq Z_2\supseteq\cdots$ (a rigorous definition requires to specify the notion of transversality, as the sets under consideration are not smooth in general). Notice that transversality to the smaller set $Z_1$ is a weaker condition than transversality to $Z_0$. Our sufficient geometric condition for exact regularity is \begin{equation}\label{eq:iterated_char_sets} Z_k\quad \text{is empty for some $k$}.\end{equation} 
An obvious obstruction to \eqref{eq:iterated_char_sets} is the existence of a \emph{characteristic submanifold} $\Sigma$, that is, a submanifold (typically with nonempty boundary) to which all vector fields $X_j$ are tangent. This is the situation of interest in this work. 

One may think of the notions just introduced in terms of dynamics: \eqref{eq:iterated_char_sets} means that one may escape from the degeneration locus in a small time, by flowing along finitely many vector fields $X_j$ (this is indeed the idea behind the proof of the cited result from \cite{DallAra2024}), while one cannot similarly escape from a generic point of a characteristic submanifold in a short time. One should notice that the presence of a characteristic submanifold does not by itself preclude the existence of paths connecting any pair of points in $M$, whose velocity vectors are at all times proportional to some $X_j$. In fact, one may be able to first flow inside the characteristic submanifold, and then escape through its boundary (see Lemma \ref{lem:berhanu_mendoza} below for a wide class of examples where this happens). As explained below, this "single Sussmann's orbit" condition (cf.~Section \ref{sec:sussmann} for this terminology) allows to formulate a basic $L^2$ existence theory for the PDE $Pu=f$, via an appropriate Poincaré inequality due to Omori and Kobayashi \cite{OmoriKobayashi1999}. In fact, these authors conjecture that the single Sussmann's orbit condition implies global hypoellipticity of $P$ (see \cite[p.~628]{OmoriKobayashi1999}) and indeed prove that global hypoellipticity holds in various examples where characteristic submanifolds, or even open sets foliated by characteristic submanifolds, exist.\newline 

The main contribution of this paper is the proof of a "scaling lemma" (Lemma \ref{lem:scaling}) showing that, if $\Sigma$ is a characteristic submanifold, then an obstruction to exact regularity of the sublaplacian $P$ is the existence of singular solutions of a Dirichlet problem for an operator $\mathfrak{P}$ defined on the normal bundle $N\Sigma$ (where the Dirichlet condition amounts to vanishing on $N(\partial M)$). For simplicity, only the case of characteristic hypersurfaces is discussed, but it is not difficult to generalize the result to higher codimension. The operator $\mathfrak{P}$ is obtained from $P$ by a familiar scaling procedure; in fact, it is known under the name of "normal operator" in the language of $b$-calculus (see, e.g., \cite{Melrose1991}). See the next section of this introduction for a discussion of antecedents of this scaling lemma in the literature. 

We illustrate the applicability of the scaling lemma by proving that a class of sublaplacians, which we call "worm sublaplacians", are not exactly regular (Theorem \ref{thm:worm_irreg}). These are defined on the three-dimensional boundary $M$ of a Diederich--Fornaess worm domain $\Omega\subseteq \mathbb{C}^2$, in terms of vector fields $X_1,X_2$ spanning the complex tangent directions to $M$. Diederich--Fornaess worm domains are famous objects of several complex variables, whose definition is recalled in Section \ref{sec:DF}. For our present PDE discussion, salient features of these examples are that $X_1,X_2$ and their first commutator $[X_1,X_2]$ span the tangent space of $M$ at all points outside a characteristic $2$-dimensional annulus $\Sigma$, and that the "single Sussmann orbit" condition mentioned above holds. As already mentioned, the scaling lemma reduces the proof of failure of exact regularity to the existence of certain singular solutions for a Dirichlet problem for the model operator $\mathfrak{P}$. The latter turns out to be a foliation Laplacian, and singular solutions are generated by the nontriviality of the holonomy of the zero section $\Sigma \subseteq N\Sigma$ (as indicated in Section \ref{sec:geometry}). This is, in another language, the cohomological nontriviality that is responsible for all "pathological" properties of Diederich--Fornaess worms in several complex variables. We see the failure of exact regularity of worm sublaplacians as an indication that the conjecture of Omori and Kobayashi cited above may be false without an additional assumption on the degeneration locus. 

\subsection{Motivations and antecedents} This work is motivated by the desire to understand and expand the scope of some deep results of Barrett \cite{Barrett1992} and Christ \cite{Christ1995} in several complex variables, in particular on the regularity theory of the $\bar\partial$-Neumann problem. While no particular knowledge of several complex variables is needed to understand the rest of our paper, beyond the most basic definitions of CR structure and strong pseudoconvexity of the boundary of a domain, here we assume some familiarity with the $\bar\partial$-Neumann problem (a standard reference is \cite{Straube2010}). 

Barrett discovered in \cite{Barrett1992} that the Bergman projection on Diederich--Fornaess worms is not exactly regular; its proof consists of a scaling-and-contradiction argument, that is the model for our argument here. Of course, the scheme of proof is by itself classical; interestingly, here one does not blow up a single point, but the whole characteristic annulus in the boundary of the worm domain. Later, Christ disproved global $C^\infty$ regularity for the Neumann operator on $(0,1)$-forms on worm domains, by ingeniously filling the substantive gap between failure of exact regularity (which follows from Barrett's theorem via \cite{BoasStraube1990}) and failure of global regularity. A couple of comments may help to put Christ's remarkable theorem in context. \begin{itemize}
	\item Previous results of Boas--Straube \cite{BoasStraube1991, BoasStraube1993} showed that the existence of a positive dimensional complex manifold $\Sigma$ in the boundary is not by itself an obstruction to exact regularity, provided a certain de Rham cohomology class $\alpha\in H^1_\mathrm{dR}(\Sigma; \R)$ is trivial (see also \cite{DallAraMongodi2023}, where $\alpha$ is called the D'Angelo class). Christ showed that such assumption is necessary: indeed, the D'Angelo class on worm domains is the class of a nonzero multiple of the angle form on the embedded annulus.  
	\item Bell and Ligocka \cite{BellLigocka1980} proved that, if $F:\Omega_1\rightarrow \Omega_2$ is a biholomorphism between smooth bounded pseudoconvex domains in $\mathbb{C}^n$ whose Bergman projections are globally regular, then $F$ extends smoothly to the boundary. Thus, Christ's counterexample highlights a serious obstacle to the $\bar\partial$-Neumann approach to the question of smooth extendibility of biholomorphisms. It is still unclear if and how it impacts the actual boundary behavior of holomorphic mappings between smooth bounded pseudoconvex domains. 
	\end{itemize}
Christ's proof combines tools of general scope, like the technique of reduction to a pseudodifferential boundary operator and other microanalytic manipulations, and other more ad hoc methods, like Fourier analysis on the annulus. This author believes that it would be of great interest to place Christ's result in an appropriate and more general setting. As a step in that direction, in this paper we bring Barrett's scaling technique to the setting of sublaplacians. The move is far from arbitrary, as the worm sublaplacians to which the methods of this paper are applied are exactly the principal parts of the boundary Kohn Laplacians $\Box_b$ on worm domains: with the notation of the previous section, \[\Box_b=(X_1-iX_2)^\dagger(X_1-iX_2)=X_1^\dagger X_1+X_2^\dagger X_2 +\text{lower order terms}.\] Thus, the phenomenon of exact irregularity already appears for the real operators $P=X_1^\dagger X_1+X_2^\dagger X_2$. While this result is new, the investigation of irregularity of real second order PDEs via scaling was already started by Christ. In his survey paper \cite{Christ1999}, Christ states that certain second order operators on the two-dimensional torus, which are elliptic away from a segment, fail to be globally regular; see \cite[Theorem 7.1]{Christ1999}, which refers to the preprint \cite{Christ1995} for a proof. This work, which seems to have never been published, transfers the original problem on the torus to a model \emph{Dirichlet} problem on a strip (that is, the normal bundle of the characteristic segment). Our scaling lemma (Lemma \ref{lem:scaling}) significantly generalizes the key insight of Christ: scaling turns the original problem on a manifold without boundary into a Dirichlet problem on the normal bundle of a characteristic submanifold. 

\subsection{Contents of the paper} In Section \ref{sec:basic}, we prove the basic theorems on existence and uniqueness of global "finite energy" solutions of sublaplacians on compact manifolds without boundary. Beyond classical Hilbert space theory, the key tools are Sussmann's theorem on orbits of families of vector fields and the observation of Omori and Kobayashi that all that is needed for the validity of the key Poincaré inequality is the existence of a single Sussmann's orbit (Lemma \ref{lem:omori_kobayashi}). A result of Berhanu and Mendoza guarantees that the latter condition is satisfied by a large class of sublaplacians originating from several complex variables (see Section \ref{sec:scv} for this). 

Once the foundations are laid, we pass to the technical heart of the paper, namely the scaling lemma (Lemma \ref{lem:scaling}). Its preparation, statement and proof occupy Section \ref{sec:scaling}. As we already mentioned, the construction of "model" vector fields and sublaplacians on the normal bundle of a characteristic submanifold is nothing but the familiar construction of normal operators associated to $b$-operators in the $b$-calculus. Notice, though, that our sublaplacians are not $b$-elliptic. E.g., worm sublaplacians are modeled by Heisenberg sublaplacians away from the characteristic annulus (commutators of length two suffice to span), and by a foliation Laplacian at the annulus. The author, who is not an expert on the matter, is not aware of any ready-made pseudodifferential calculus that may accomodate it. 

As a technical remark, we would like to point out the trick of extending the characteristic manifold $\Sigma$ to a larger submanifold $\Sigma_0$ by flowing along a combination of the original vectors $X_j$ that is everywhere transversal to the boundary of $\Sigma$ (see Section \ref{sec:characteristic}). This is the key to the proof that the solutions of the model problem satisfy the Dirichlet boundary condition, thus generalizing Christ's insight to our geometric setting. 

In Section \ref{sec:worm} we establish the failure of exact regularity for worm sublaplacians. What is of interest here is the geometry of the model operator $\mathfrak{P}$ that the scaling lemma lends us: $\mathfrak{P}$ can be described as a Laplacian along a foliation on the normal bundle $N\Sigma$. The nontriviality of the D'Angelo class of the annulus corresponds to nontriviality of the holonomy of this foliation at the zero section, which is the unique compact leaf. A more conceptual explanation for the appearance of a foliation is given in Section \ref{sec:geometry}. Let us finally mention that the existence of singular solutions for the model Dirichlet problem is obtained as a corollary of a complete asymptotic expansion of solutions corresponding to data supported on one side of the zero section (Proposition \ref{prp:asymptotic}). 

\subsection{Acknowledgments} I would like to thank Alessio Martini for many discussions on various topics related to the content of this paper. I acknowledge the support of Istituto Nazionale di Alta Matematica "F. Severi". 

\section{Existence of weak solutions}\label{sec:basic}

\subsection{Setup}\label{sec:setup}

Our basic data are the following: 
\begin{itemize}
\item A closed (i.e., compact and without boundary) connected smooth manifold $M$.
\item A \emph{system of vector fields}, that is, a finite collection $\mathcal{X}=\{X_1,\ldots, X_k\}$ of real smooth vector fields, globally defined on $M$. 
\item A measure $\mu$ of total mass $1$ which, when expressed in local coordinates, has a smooth positive density with respect to Lebesgue measure. Equivalently, one may think of $\mu$ as a smooth positive $1$-density on $M$. We say that $\mu$ is a \emph{normalized smooth positive measure}.
	\end{itemize}

To a triple $(M,\mathcal{X}, \mu)$ as above there corresponds an \emph{energy space}, consisting of square-integrable functions on $M$ whose derivatives along elements of $\mathcal{X}$ are also square-integrable. In order to give a precise definition, we rapidly recall a few standard facts and introduce some useful notation.\newline 

For $L$ a linear partial differential operator on $M$, we denote by $L^\dagger$ its formal adjoint with respect to the measure $\mu$, that is, the unique linear partial differential operator satisfying the identity \[
\int_M Lf\cdot g\, d\mu = \int_M f\cdot L^\dagger g\, d\mu
\] for all $f,g\in C^\infty(M)$. As usual, the above identity allows us to define $Lf$ for any distribution $f$. Here we define a distribution as a continuous linear functional on $C^\infty(M)$ (with its usual Fréchet topology). Taking advantage of the background measure, we can identify any integrable function $f\in L^1(M, \mu)$ with the distribution $g\mapsto \int_Mgf\, d\mu$. In particular, $Lf$ is defined as a distribution for every integrable $f$.

Let $\mathcal{X}$ be a system of vector fields. We define the corresponding \emph{energy space} by \[
\mathcal{D}_\mathcal{X}:=\{f\in L^2(M,\mu)\colon\, X_jf\in L^2(M,\mu)\quad \forall j\},
\]
equipped with the graph norm \begin{equation}\label{eq:graph_norm}
\lVert f\rVert_{\mathcal{D}_\mathcal{X}}:=\left(\lVert f\rVert_{L^2(M, \mu)}^2+\sum_j\lVert X_jf\rVert_{L^2(M, \mu)}^2\right)^\frac{1}{2}.
\end{equation}
The following proposition is classical. 

\begin{prp}[Density of test functions]\label{prp:density}
 The energy space $\mathcal{D}_\mathcal{X}$ is complete in the graph norm \eqref{eq:graph_norm}, and the subspace $C^\infty(M)$ is dense in this norm. 
	\end{prp}

\begin{proof}
The first half is an easy consequence of the continuity of $X_j$ as operators on the space of distributions. The second half follows immediately from Friedrichs' lemma (see, e.g., \cite[Lemma 25.4]{Treves1975} for a local formulation, from which the analogous statement on compact manifolds follows by a standard partition of unity argument). 
	\end{proof}
The next definition will be the focus of our attention in the rest of the section. 

\begin{dfn}\label{def:global_poincare}
We say that the global $L^2$ Poincaré inequality holds on $(M,\mathcal{X},\mu)$ if there exists a constant $c>0$ such that \begin{equation}\label{eq:global_poincaré}
\sum_j\lVert X_j f\rVert_{L^2(M,\mu)}^2\geq c \lVert f-\mathrm{av}_\mu(f)\rVert_{L^2(M,\mu)}^2
\end{equation} for every $f\in \mathcal{D}_\mathcal{X}$, where $\mathrm{av}_\mu(f):=\int_M f\, d\mu$ is the average of $f$. 
	\end{dfn}

Notice that the validity of the global $L^2$ Poincaré inequality depends only on the vector fields, and not on the measure. To see this, one may use the identity \[
\lVert f-\mathrm{av}_\mu(f)\rVert_{L^2(M,\mu)}^2 = \min_{a\in \mathbb{C}}\lVert f-a\rVert_{L^2(M,\mu)}^2
\] and the comparability of $L^2$ norms corresponding to different smooth positive measures.\newline 

As shown by Omori--Kobayashi \cite{OmoriKobayashi1999}, the global $L^2$ Poincaré inequality holds under the rather weak assumption that the group of diffeomorphisms generated by the flows of the vector fields $X_j$ acts transitively on $M$. For a proper discussion of this result, we first need to recall the elements of Sussmann's theory of orbits. 

\subsection{Sussmann orbits}\label{sec:sussmann} Let $\mathcal{X}=\{X_1,\ldots, X_k\}$ be a system of vector fields on $M$. We denote by $G_\mathcal{X}$ the group of diffeomorphisms of the form \begin{equation}\label{eq:flow_composition}
\exp(t_N X_{j_N})\circ \cdots \circ \exp(t_1 X_{j_1}), 
\end{equation}
for any choice of indices $j_1,\ldots, j_N\in \{1,\ldots, k\}$ and "times" $t_1,\ldots, t_N\in \R$. Here $\exp(tX)$ is the flow at time $t$ of the vector field $X$. Recall that $M$ is compact, so flows exist for all times (Sussmann's theory applies to the case of non-compact manifolds as well, although some care is needed in its formulation, because flows of smooth vector fields need not exist for all times). Following \cite{Sussmann1973}, we define the \emph{Sussmann orbits} of $\mathcal{X}$ as the orbits of the action of $G_\mathcal{X}$ on $M$. That is, two points are in the same orbit if they are connected by a piecewise smooth path consisting of integral curves of vector fields in $\mathcal{X}$. Denote by $\widetilde{\mathcal{X}}$ the set of vector fields of the form $\Phi_*X$, where $\Phi\in G_\mathcal{X}$, $X\in \mathcal{X}$, and the lower asterisk denotes the push-forward of a vector field along a diffeomorphism. We have the following crucial result (see \cite[Theorem 4.1]{Sussmann1973} for a more complete statement). 

\begin{thm}[Sussmann's Theorem]\label{thm:sussmann}
	The Sussmann orbits of $\mathcal{X}$ are immersed submanifolds whose tangent space at any point is spanned by the elements of $\widetilde{\mathcal{X}}$. 	
\end{thm}

The following corollary is almost immediate. 

\begin{cor}\label{cor:sussmann}
The group $G_\mathcal{X}$ acts transitively on $M$ (that is, there is a single Sussmann orbit)  if and only if the elements of $\widetilde{\mathcal{X}}$ span the tangent space of $M$ at each point.	
	\end{cor}

\begin{proof}
The ``only if'' part is obvious. If $\widetilde{\mathcal{X}}$ spans the tangent space, by Sussmann's theorem all orbits are open. By connectedness, there must be only one orbit. 
	\end{proof}

The above theorem reduces the transitivity of $G_\mathcal{X}$ to the infinitesimal transitivity of a global object, namely the expanded set of vector fields $\widetilde{\mathcal{X}}$. Sussmann also showed that the distribution pointwise spanned by $\widetilde{\mathcal{X}}$ is involutive, see \cite[Theorem 4.1; part (d)]{Sussmann1973}, so in particular the condition of Corollary \ref{cor:sussmann} for the existence of a single orbit is satisfied when the system of vector fields $\mathcal{X}$ satisfies Hörmander's commutator condition, that is, when the vector fields $X_j$ together with all their iterated commutators of all orders span the tangent space of $M$ at each of its points. This recovers the well-known Chow--Rashevsky Theorem as a (very) special case of Corollary \ref{cor:sussmann}. 

\subsection{Global $L^2$ Poincaré inequality}

As anticipated above, we have the following result, whose a priori formulation appears in \cite{OmoriKobayashi1999}. 

\begin{lem}[Omori--Kobayashi]\label{lem:omori_kobayashi}
	Let $(M, \mathcal{X}, \mu)$ be as in Section \ref{sec:setup}, with $\mathcal{X}=\{X_1,\ldots, X_k\}$. If $G_\mathcal{X}$ acts transitively on $M$, then the global $L^2$ Poincaré inequality holds on $(M,\mathcal{X}, \mu)$, that is, there exists $c>0$ such that \[
	\sum_j \lVert X_j f\rVert_{L^2(M,\mu)}^2\geq c\lVert f-\mathrm{av}_\mu(f)\rVert_{L^2(M,\mu)}^2	\qquad \forall f\in \mathcal{D}_\mathcal{X}.\] 
	\end{lem}

\begin{proof} By Proposition \ref{prp:density}, it is enough to prove the estimate for $f\in C^\infty(M)$. This is the content of \cite[Theorem 2.1]{OmoriKobayashi1999}. We reproduce the proof for the sake of completeness. By the single-orbit case of Sussmann's Theorem (Corollary \ref{cor:sussmann}) and compactness of $M$, there exist $Y_1,\ldots, Y_N\in \widetilde{\mathcal{X}}$ that span the tangent space of $M$ at each point. We have the standard global ellipticity bound \[
	\lVert f-\mathrm{av}_\mu(f)\rVert_{L^2(M,\mu)}^2\leq C\sum_\ell\lVert Y_\ell f\rVert_{H^{-1}}^2, 
	\]
	where $C$ is a constant and $\lVert \cdot \rVert_{H^s}$ denotes the $L^2$-Sobolev norm of order $s$. Our goal now boils down to proving that for every $Y\in \widetilde{\mathcal{X}}$ there exists $C_Y>0$ such that \begin{equation}\label{eq:omori_kobayashi_induction}
		\lVert Yf\rVert_{H^{-1}}^2\leq C_Y\sum_j \lVert X_j f\rVert_{L^2(M,\mu)}^2\qquad \forall f\in C^\infty(M).\end{equation} 
	Let $Y=\Phi_*X$, where $\Phi\in G_\mathcal{X}$ and $X\in \mathcal{X}$. Since $\Phi$ is the composition of $N$ flows (as in \eqref{eq:flow_composition}), we may argue by induction on $N$. The case $N=0$ is trivial. Assume then that $Y=\exp(tX')_*Y'$, where $t\in \R$, $X'\in \mathcal{X}$, and inequality \eqref{eq:omori_kobayashi_induction} holds with $Y'$ in place of $Y$. We use the identity \[
	\exp(tX')_*Y'=\exp(-tX')^*\circ Y'\circ \exp(tX')^*, 
	\]
	where we view vector fields as operators on functions, and the upper asterisk denotes the pull-back operator (on functions). Since pull-backs along diffeomorphisms are bounded on Sobolev spaces of every order, it is enough to prove \eqref{eq:omori_kobayashi_induction} with $Y'\circ \exp(tX')^*$ in place of $Y$. By the Fundamental Theorem of Calculus, \[
	Y'\circ \exp(tX')^*f = Y'f+\int_0^t Y'\circ \exp(sX')^*(X'f)\, ds, 
	\]
	and by sub-additivity of the $H^{-1}$ norm, \begin{equation}\label{eq:omori_kobayashi_induction_2}
	\lVert Y'\circ \exp(tX')^*f\rVert_{H^{-1}}\leq \lVert Y'f\rVert_{H^{-1}}+ \int_{\min\{0,t\}}^{\max\{0,t\}}\lVert Y'\circ \exp(sX')^*(X'f)\rVert_{H^{-1}}\, ds.
	\end{equation}
	By the inductive assumption, $\lVert Y'f\rVert_{H^{-1}}^2\leq C_{Y'}\sum_j \lVert X_j f\rVert_{L^2(M,\mu)}^2$, so we are left to deal with the second term on the right hand side. 
	
	Using again that pull-backs preserve the Sobolev scale, plus the fact that $Y'$ is bounded from $H^0$ to $H^{0-1}$, we see that $\lVert Y'\circ \exp(sX')^*(X'f)\rVert_{H^{-1}}\leq C(s)\lVert X'f\rVert_{H^0}$. Here the constant $C(s)$ depends on finitely many $C^k$ norms of the components of $\exp(sX')$, computed in an appropriate finite collection of local charts. Thus, $C(s)$ remains bounded as $s$ varies in the compact interval between $0$ and $t$. The $H^0$ and $L^2(M,\mu)$ norms are equivalent, so the integral in \eqref{eq:omori_kobayashi_induction_2} is bounded by a constant times $\lVert X'f\rVert_{L^2(M,\mu)}$. Since $X'\in \mathcal{X}$, the right hand side of \eqref{eq:omori_kobayashi_induction_2} is bounded by a constant times $\left(\sum_j \lVert X_j f\rVert_{L^2(M,\mu)}^2\right)^\frac{1}{2}$. This completes the inductive step, and the proof of the lemma. \end{proof}

\subsection{Global weak solutions of sum-of-squares operators}\label{sec:global_solutions}

We continue working on a closed connected smooth manifold $M$, equipped with a normalized smooth positive measure $\mu$ and a system of vector fields $\mathcal{X}=\{X_1,\ldots, X_k\}$. From this data, we may define the sum-of-squares operator \[
P=\sum_j X_j^\dagger X_j. 
\]
By standard Hilbert space methods, the global $L^2$ Poincaré inequality \eqref{eq:global_poincaré} yields existence of weak solutions in the energy space for the operator $P$. This is recorded in the next result. 
 
\begin{thm}[Existence and uniqueness of global finite energy solutions]\label{thm:global_solutions}
Assume that $(M, \mathcal{X}, \mu)$	satisfies the global $L^2$ Poincaré inequality. Let $f\in L^2(M,\mu)$. The equation \[Pu=f\] admits a (distributional) solution $u\in \mathcal{D}_\mathcal{X}$ if and only if \[
\mathrm{av}_\mu(f)=\int_M f\, d\mu=0.
\] In this case, there exists a unique solution $u_0\in \mathcal{D}_\mathcal{X}$ of zero average and the general solution is of the form $u=u_0+k$, where $k\in \mathbb{C}$. Moreover, \begin{equation}\label{eq:energy_estimate}
\lVert u_0\rVert_{L^2(M,\mu)}^2\leq c^{-2}\lVert f\rVert_{L^2(M,\mu)}^2, \quad \sum_j\lVert X_ju_0\rVert_{L^2(M,\mu)}^2\leq c^{-1}\lVert f\rVert_{L^2(M,\mu)}^2, 
\end{equation}
where $c$ is the constant appearing in the global $L^2$ Poincaré inequality. 
	\end{thm}

\begin{proof}
The Hermitian form \[q(u,v):=\sum_j \int_MX_ju\cdot \overline{X_j v}\, d\mu\] is well-defined on $\mathcal{D}_\mathcal{X}$. The global $L^2$ Poincaré inequality may be rewritten as \begin{equation}\label{eq:global_poincaré_2}
q(u,u) \geq c\lVert u-\mathrm{av}_\mu(u)\rVert_{L^2(M,\mu)}^2
\end{equation} for all $u\in \mathcal{D}_\mathcal{X}$ and some $c>0$. Thus, $q$ is positive definite and comparable with the square of the graph norm \eqref{eq:graph_norm} on the subspace \[
\mathcal{D}_\mathcal{X}^0:=\{u\in \mathcal{D}_\mathcal{X}\colon\, \mathrm{av}_\mu(u)=0\}.
\]
Since $\mathcal{D}_\mathcal{X}^0$ is closed with respect to the graph norm \eqref{eq:graph_norm}, it follows that $\mathcal{D}_\mathcal{X}^0$ is a Hilbert space with respect to $q$. Let $f\in L^2(M,\mu)$. By \eqref{eq:global_poincaré_2} and Cauchy--Schwarz, the anti-linear form $v\mapsto \int_Mf\overline{v}\, d\mu$ is bounded on $\mathcal{D}_\mathcal{X}^0$ with norm $\leq c^{-\frac{1}{2}}\lVert f\rVert_{L^2(M,\mu)}$. By the Riesz representation theorem, there exists a unique $u_0\in \mathcal{D}_\mathcal{X}^0$ such that \begin{equation}\label{eq:weak_equation}
q(u_0,v)= \int_Mf\overline{v} \, d\mu\qquad 
\end{equation} for all $v\in \mathcal{D}_{\mathcal{X}}^0$, with accompanying estimate \begin{equation}\label{eq:riesz_1}
q(u_0,u_0)\leq c^{-1}\lVert f\rVert_{L^2(M,\mu)}^2. 
\end{equation} Combining this with \eqref{eq:global_poincaré_2}, we obtain \begin{equation}\label{eq:riesz_2}
\lVert u_0\rVert_{L^2(M,\mu)}^2\leq c^{-2}\lVert f\rVert_{L^2(M,\mu)}^2. 
\end{equation}
If the datum $f$ has zero average, \eqref{eq:weak_equation} trivially holds for $v$ constant, and hence for all $v\in \mathcal{D}_\mathcal{X}$. Since $u_0\in\mathcal{D}_\mathcal{X}$, the validity of \eqref{eq:weak_equation} for all $v\in \mathcal{D}_\mathcal{X}$ is equivalent to the distributional identity $Pu_0=f$ (here one uses Proposition \ref{prp:density} to pass from test functions to general $v\in \mathcal{D}_\mathcal{X}$). On the other hand, testing \eqref{eq:weak_equation} on $v\equiv 1$, we find that the condition $\mathrm{av}_\mu(f)=0$ is necessary for the existence of finite energy solutions. Finally, uniqueness of solutions up to constants, or equivalently constancy of solutions of the homogeneous equation $Pu=0$, follows immediately from the global $L^2$ Poincaré inequality. 
	\end{proof}

In view of Lemma \ref{lem:omori_kobayashi}, we have the following corollary. 

\begin{cor}\label{cor:global_solutions}
		Assume that $M$ is a closed connected manifold, equipped with a normalized smooth positive measure $\mu$ and a system of vector fields $\mathcal{X}=\{X_1,\ldots, X_k\}$. If $G_\mathcal{X}$ acts transitively on $M$, then the conclusions of Theorem \ref{thm:global_solutions} hold true. 
	\end{cor}

\subsection{Examples from several complex variables}\label{sec:scv} We now show that analysis in several complex variables provides a natural source of sum-of-squares operators to which the theory discussed above applies. 

Let $\Omega\subseteq \mathbb{C}^n$ be a smooth domain. Then the boundary $M=\partial\Omega$ is a smooth real hypersurface of $\mathbb{C}^n$. As such, it comes equipped with a natural rank $2n-2$ sub-bundle of its tangent bundle $TM$, namely \begin{equation}\label{eq:horizontal_bundle}
HM=TM\cap JTM, 
\end{equation}
where $J$ is the endomorphism acting as multiplication by $i$ on tangent vectors of $\mathbb{C}^n$. We say that a system of vector fields $\mathcal{X}=\{X_1,\ldots, X_k\}$ generates $HM$ if they span $HM$ at each point of $M$ (in particular, $k$ may be larger than $2n-2$). We have the following result of Berhanu--Mendoza (\cite[Corollary 4.6]{BerhanuMendoza1997}). For the reader's convenience, we extract from their paper the elegant proof of this particular statement. 

\begin{lem}[Berhanu--Mendoza]\label{lem:berhanu_mendoza} Let $n\geq 2$. Let $\Omega\subseteq \mathbb{C}^n$ be a smooth bounded domain with connected boundary $M$. Let $\mathcal{X}$ be a system of vector fields generating $HM$. Then $G_\mathcal{X}$ acts transitively on $M$. 
	\end{lem}

\begin{proof}
	Since $HM$ has co-rank one, by Theorem \ref{thm:sussmann} the Sussmann orbits are either open or of codimension one in $M$. In the latter case, their tangent spaces must coincide with the fibers of $HM$. Since $HM$ is $J$-invariant, codimension one orbits are positive-dimensional complex submanifolds of $\mathbb{C}^n$. Thus, $M=M_0\cup M_1$, where $M_0$ is the union of the open orbits and $M_1$ is a union of positive-dimensional complex manifolds. 
	
	The thesis is that $M_1$ is empty: if every orbit is open, by connectedness there must be only one. We argue by contradiction, assuming that $M_1$ is non-empty. Since $M_1$ is compact, the function $\phi(z)=\sum_{j=1}^n|z_j|^2$, restricted to $M_1$, achieves a maximum at a point $p$. If $L_p$ is the orbit through $p$, the restriction of $\phi(z)$ to $L_p$ is a strictly plurisubharmonic function on a complex manifold of positive dimension admitting a maximum. This is the desired contradiction. 
	\end{proof}

Notice that no complex-geometric assumption is made on $\Omega$ (e.g., it need not be pseudoconvex). Combining Lemma \ref{lem:berhanu_mendoza} and Corollary \ref{cor:global_solutions}, we obtain the following theorem. 

\begin{thm}\label{thm:global_solutions_cr}
Let $n\geq 2$. Let $\Omega\subseteq \mathbb{C}^n$ be a smooth bounded domain with connected boundary $M$. Let $\mathcal{X}=\{X_1,\ldots, X_k\}$ be a system of vector fields generating $HM$ and let $\mu$ be a normalized smooth positive measure on $M$. Then the conclusions of Theorem \ref{thm:global_solutions} apply to the equation $\sum_jX_j^\dagger X_ju=f$ (adjoints are with respect to $\mu$). 
	\end{thm}

\subsection{Exact regularity}\label{sec:exact_regularity} We continue working in the same setting as in Section \ref{sec:global_solutions}. Assume that $(M,\mathcal{X}, \mu)$ satisfies the global $L^2$ Poincaré inequality, so that Theorem \ref{thm:global_solutions} applies to $P=\sum_j X_j^\dagger X_j$. Denote by $P^{-1}$ the solution operator mapping a zero average datum $f\in L^2(M,\mu)$ to the unique zero average $u\in \mathcal{D}_\mathcal{X}$ such that $Pu=f$. By \eqref{eq:energy_estimate}, $P^{-1}$ is bounded on $L^2(M,\mu)$. 

We now define precisely the property of utmost interest for us. Denote by $H^k(M)$ the $L^2$ Sobolev space of order $k\in \mathbb{N}$ on $M$. Recall that these are intrinsically defined normable spaces. 
We fix once and for all norms $\lVert \cdot \rVert_{H^k(M)}$.

\begin{dfn}[Exact regularity]\label{def:exact_reg}
	We say that $P$ is exactly regular of order $k$ (on the $L^2$ Sobolev scale) if the following holds: for any zero average datum $f\in H^k(M)$, we have $P^{-1}f\in H^k(M)$. 
\end{dfn}

It is worth remarking that exact regularity is a \emph{global} regularity property. For example, exact regularity of all orders implies (via Sobolev embedding) smoothness of solutions when the datum is smooth, but it does not rule out propagation of singularities: smoothness of the datum on an open set does not necessarily imply smoothness of the solution on the same open set. 

We conclude by mentioning a couple of useful facts. \begin{enumerate}
\item If $P$ is exactly regular of order $k$, then $P^{-1}:H^k(M)\cap\{\mathrm{av}_\mu(f)=0\}\rightarrow H^k(M)$ is a well-defined \emph{closed} operator. By the closed-graph theorem, it is bounded. This means that there exists a constant $C$ such that   \[
\lVert u\rVert_{H^k(M)}\leq C\lVert f\rVert_{H^k(M)}
\] for all $f\in H^k(M)$ of zero average. 
\item If $P$ is exactly regular of order $k$, then a standard interpolation argument shows that $P$ is exactly regular of any order $j\leq k$ (this fact is not used in the sequel).
	\end{enumerate}

\section{The scaling argument}\label{sec:scaling}

\subsection{Characteristic submanifolds, grazing points, and flow collars}\label{sec:characteristic} Let $M$ be an ambient smooth manifold without boundary. We recall some standard terminology, mostly in order to fix notation. 

A subset $\Sigma\subseteq M$ is said to be a \emph{$d$-dimensional embedded submanifold} if for every $p\in \Sigma$, one of the following occurs: either there are local coordinates $(x,t)\in \R^d\times \R^{n-d}$ centered at $p$ (that is, $(x(p),t(p))=(0,0)$) such that $\Sigma=\{t=0\}$ near $p$; or there are local coordinates $(x',r,t)\in \R^{d-1}\times \R\times \R^{n-d}$ centered at $p$ such that $\Sigma=\{t=0, r\leq 0\}$ near $p$. In the latter case, we say that $p$ is a boundary point of $\Sigma$. We denote by $\partial\Sigma$ the set of all boundary points, and by $\overset{\circ}{\Sigma}:=\Sigma\setminus \partial \Sigma$ the interior of $\Sigma$. The boundary $\partial \Sigma$ is of course a $(d-1)$-dimensional embedded submanifold with empty boundary. E.g., the disc $\Sigma=\{x^2+y^2\leq 1, z=0\}$ is a (compact) $2$-dimensional embedded submanifold of $\R^3_{x,y,z}$, whose boundary is a circle. For brevity, we call \emph{adapted coordinates at $p\in \Sigma$} any system of coordinates $(x,t)$ or $(x',r,t)$ as above. 

Fix a system of vector fields $\mathcal{X}=\{X_1,\ldots, X_k\}$ on $M$. The following notion is crucial for the rest of the paper. 
\begin{dfn}[Characteristic submanifolds] An embedded submanifold $\Sigma\subseteq M$ is said to be characteristic (with respect to $\mathcal{X}$) if every vector field $X_j\in\mathcal{X}$ is tangent to $\Sigma$. \end{dfn}
Notice that this is a good definition, because $\Sigma$ has a well-defined tangent space at every point, including at the boundary. In fact, since the elements of $\mathcal{X}$ are continuous, tangency to $\Sigma$ is equivalent to tangency to its interior $\overset{\circ}{\Sigma}$. 

If $\Sigma$ is characteristic, then an interior point $p\in\overset{\circ}{\Sigma}$ cannot leave $\Sigma$ by short-time flows along vector fields in $\mathcal{X}$, but the same need not be true at boundary points, or for long-time flows. The next definition deals exactly with this matter. 

\begin{dfn}[Grazing points] Let $\Sigma$ be a characteristic embedded submanifold. A boundary point $p\in \partial \Sigma$ is said to be a grazing point (with respect to $\mathcal{X}$) if each $X_j\in \mathcal{X}$ is tangent to $\partial \Sigma$ at $p$. \end{dfn}

If $p\in \overset{\circ}{\Sigma}$ and $(x,t)$ are adapted coordinates at $p$, then the condition that $\Sigma$ is characteristic means that for each $j$ we have \[
X_j=c_j(x,t)\cdot \nabla_x+d_j(x,t)\cdot \nabla_t
\] with \[
d_j(x,0)\equiv 0. 
\]
Here the coefficients $c_j$ and $d_j$ are respectively $\R^d$-valued and $\R^{n-d}$-valued, and $\nabla_x$, $\nabla_t$ denote gradients in the indicated subsets of coordinates. 

If $p\in \partial \Sigma$ and $(x',r,t)$ are adapted coordinates at $p$, then the condition that $p$ is a grazing point consists in the requirement that for all $j$ we have 
\begin{equation}\label{eq:vector_field_boundary}
X_j=c_j(x',r,t)\cdot \nabla_{x'}+\widetilde{c_j}(x',r,t)\partial_r+d_j(x',r,t)\cdot \nabla_t
\end{equation} with \[
d_j(x',r,0)\equiv 0\qquad \textrm{if}\quad r\leq 0
\] and \[
\widetilde{c_j}(0,0,0)= 0.
\]

For later reference, we isolate a useful fact (stated here for compact submanifolds, but  a minor modification of the statement holds in general). 

\begin{prp}[Flow collar]\label{prp:flow_collar} Let $M$ be a smooth manifold without boundary, equipped with a system of vector fields $\mathcal{X}=\{X_1,\ldots, X_k\}$. 
	Let $\Sigma$ be a characteristic embedded submanifold of dimension $d$. Assume that $\Sigma$ is compact and has no boundary grazing point. Then there exist an embedded submanifold $\Sigma_0$ of dimension $d$ and a vector field $Y=\sum_j b_jX_j$, with $b_j\in C^\infty(M)$, such that the following properties hold:\begin{enumerate}
		\item $\Sigma_0$ has empty boundary and contains $\Sigma$.
		\item $Y$ is tangent to $\Sigma_0$. 
		\item $Y$ is transversal to $\partial \Sigma$ at each of its points.
		\end{enumerate}
	\end{prp}
\begin{proof} If $\Sigma$ has empty boundary, the statement (and the argument below) trivialize, so one may assume that $\partial\Sigma$ is non-empty. 

	We say that a vector field of the form $Y=\sum_j b_jX_j$, with $b_j\in C^\infty(M)$, is \emph{outward pointing} at $p\in \partial \Sigma$ if, in adapted coordinates $(x',r,t)$ at $p$, we have $Y=c(x',r,t)\cdot \nabla_{x'}+\widetilde{c}(x',r,t)\partial_r+d(x',r,t)\cdot \nabla_t$ with $\widetilde{c}(0,0,0)>0$. 
	
	Since no boundary point $p\in\partial \Sigma$ is grazing, for every $p\in \partial \Sigma$ there exists $X_j\in \mathcal{X}$ such that either $X_j$ or $-X_j$ is outward pointing at $p$. Using a suitable partition of unity, now one can sum a finite number of such vector fields and obtain $Y=\sum_j b_jX_j$ that is \emph{outward-pointing} at every boundary point. In particular, $Y$ is transversal to the boundary, and \emph{(3)} is proved. Since $\Sigma$ is characteristic, $Y$ is also tangent to $\Sigma$. 

Now we use this outward pointing $Y$ to slightly extend $\Sigma$ via a standard "flow collar" construction. Since $\partial \Sigma$ is compact, there exists $\delta>0$ such that the map \[
\partial \Sigma\times (-\delta, \delta) \ni (p,s) \longmapsto \Psi(p,s):=\exp(sY)p\in M
\] 
is an embedding; denote by $\Sigma'=\Psi(\partial \Sigma\times (-\delta, \delta))$ its image, a $d$-dimensional submanifold without boundary. Taking a possibly smaller value of $\delta$ we may guarantee that $\Psi(\partial\Sigma\times (-\delta,0])\subseteq \Sigma$ and that $\Psi(\partial\Sigma\times (0, \delta))$ does not meet $\Sigma$. Thus, \[\Sigma_0:=\Sigma\cup \Psi(\partial\Sigma\times (-\delta, \delta))\] is an embedded submanifold with empty boundary, and \emph{(1)} is established. By construction, $\Sigma_0\setminus \Sigma$ is a union of integral curves of $Y$, so \emph{(2)} holds too. 
\end{proof}
Notice that in general one cannot expect the manifold $\Sigma_0$ of Proposition \ref{prp:flow_collar} to be characteristic with respect to $\mathcal{X}$. 

\subsection{Vertical Sobolev spaces on normal bundles}\label{sec:tubular} In this and the following section, we fix: \begin{itemize}
	\item A smooth manifold $M$. 
	\item An embedded hypersurface without boundary $\Sigma_1\subseteq M$. 
\end{itemize}

The hypersurface $\Sigma_1$ has no a priori relations with the submanifolds $\Sigma$ and $\Sigma_0$ of the previous sections. Let \[
N\Sigma_1=\{(p,v)\colon\, p\in \Sigma_1, \quad v\in T_pM/T_p\Sigma_1\}
\] be the normal bundle of $\Sigma_1$. Since the latter has codimension one in $M$, $N\Sigma_1$ is a real line bundle. 

To formulate the scaling lemma, we need "vertical Sobolev spaces" on $N\Sigma_1$. We recall the construction. If $\pi(p,v)=p$ is the canonical projection on $N\Sigma_1$, a local trivialization of the normal bundle consists of an open set $B\subseteq \Sigma_1$ and a diffeomorphism $\psi: \pi^{-1}(B)\rightarrow B\times \R$ compatible with the vector bundle structure, that is, such that: $\pi_1\circ \psi=\pi$ (where $\pi_1: B\times \R\rightarrow B$ is the first projection) and $\psi$ induces linear isomorphisms $\pi^{-1}(\{p\})\rightarrow \pi_1^{-1}(\{p\})\equiv \R$ for each $p\in B$. If $B$ is the domain of a system of local coordinates $(x_1,\ldots, x_{n-1})$ on $\Sigma_1$, then adjoining $t=\pi_2\circ \psi$ (where $\pi_2: B\times \R\rightarrow \R$ is the second projection) we obtain a system of bundle coordinates $(x,t)=(x_1,\ldots, x_{n-1}, t)$ on $\pi^{-1}(B)\subseteq N\Sigma_1$. On the domain of these coordinates, we have the "vertical differentiation operator" $\frac{\partial}{\partial t}$. For $k\in \mathbb{N}$, we define \[H^k_\mathrm{vert}(N\Sigma_1)\]
as the space of functions $f:N\Sigma_1\rightarrow \mathbb{C}$ (modulo almost everywhere equivalence) with the property that, for every system of local trivializing coordinates $(x,t)$ as above and every compact $K\subseteq B$, we have \begin{equation}\label{eq:vertical_sobolev}
\int_{K\times \R} \left|\frac{\partial^j f}{\partial t^j}\right|^2\, dx\, dt<+\infty\qquad \forall j\leq k. 
\end{equation}
Since transition functions between different local trivializations are linear in the fiber variable, it is enough to verify \eqref{eq:vertical_sobolev} for a collection of trivializations whose base domains $B$ cover $\Sigma_1$. 

\subsection{Tubular neighborhoods and scalings}\label{sec:dilation_machinery} We continue with the setting of the previous section. By the tubular neighborhood theorem, there exists an open neighborhood $U$ of $\Sigma_1$ in $M$, an open neighborhood $V$ of the zero section in $N\Sigma_1$, and a diffeomorphism \begin{equation}\label{eq:tubular_neighborhood}
	\Phi:U\rightarrow V
\end{equation} with the property that $\Phi(p)=(p, 0)$ for all $p\in \Sigma_1$, i.e., $\Phi$ identifies the hypersurface $\Sigma_1$ with the zero section of the normal bundle. In what follows, we assume $U$, $V$ and $\Phi$ fixed. 

By linearity of the fibers of the normal bundle, we have the one-parameter group of dilations  \[
\mathrm{Dil}_\epsilon f(p,v)=f(p,\epsilon^{-1} v)\qquad (\epsilon>0)
\] acting on functions $f:N\Sigma_1\rightarrow \mathbb{C}$. Notice that dilations of parameter $\epsilon$ transform functions (vertically) supported "at unit scale" to functions supported "at scale $\epsilon$". Using dilations, we can perform the following basic operations: \begin{enumerate}
	\item Start with a compactly supported function $g$ on the normal bundle $N\Sigma_1$.
	\item Rescale $g$: for $\epsilon$ sufficiently small, the function $\mathrm{Dil}_\epsilon g$ is supported on $V$. 
	\item  "Transplant" $\mathrm{Dil}_\epsilon g$ to the manifold $M$, that is, consider the function $\mathrm{Dil}_\epsilon g\circ \Phi$, extended by zero outside $U$. 
\end{enumerate}
Our next proposition shows what this process does to measures in the small $\epsilon $ limit. Denote by $L^1_\mathrm{comp}(N)$ the space of compactly supported integrable functions on a manifold $N$. Notice that this is well-defined, independently of the choice of a smooth positive measure on $N$. A measure $\lambda$ on $N\Sigma_1$ is said to be \emph{dilation-invariant} if \begin{equation}\label{eq:dilation_invariance}
\epsilon^{-1} \int_{N\Sigma_1}  \mathrm{Dil}_\epsilon g\, d\lambda = \int_{N\Sigma_1} g\, d\lambda
\end{equation} for all $g\in L^1(N\Sigma_1, \lambda)$.

\begin{prp}[Scaling measures]\label{prp:scaling_measures} Let $\mu$ be a smooth positive measure on $M$. There exists a unique smooth positive measure $\lambda$ on $N\Sigma_1$ with the property that \begin{equation}\label{eq:lambda_def}
		\lim_{\epsilon\rightarrow 0}\epsilon^{-1}\int_M (\mathrm{Dil}_\epsilon g)\circ \Phi\, d\mu = \int_{N\Sigma_1} g\, d\lambda
	\end{equation} for all $g\in L^1_\mathrm{comp}(N\Sigma_1)$. The measure $\lambda$ is dilation-invariant. 

More generally, if $g_j$ converges weakly in $L^1_\mathrm{comp}(N\Sigma_1)$ to $g$  (that is, the functions $g_j$ and $g$ are supported on a common compact set and $g_j$ converges weakly in $L^1$ to $g$ on that compact set) and $\epsilon_j\rightarrow 0$, then \begin{equation}\label{eq:lambda_def_upgrade}
	\lim_{j\rightarrow +\infty}\epsilon_j^{-1}\int_M (\mathrm{Dil}_{\epsilon_j} g_j)\circ \Phi\, d\mu = \int_{N\Sigma_1} g\, d\lambda.
\end{equation}
	\end{prp}
\begin{proof} Assume first that $g\in L^1_\mathrm{comp}(N\Sigma_1)$ is supported in the domain of a chart trivializing the normal bundle $N\Sigma_1$. Denote by $(x,t)\in \R^{n-1}\times \R$ the local coordinates on such a chart, and identify functions and measures with their representations in these coordinates. The push-forward of the measure $\mu$ via $\Phi$ has the form $\rho(x,t)dx\, dt$, where $\rho$ is smooth and positive. Hence, the left hand side of \eqref{eq:lambda_def} can be written as \begin{eqnarray*}
	\lim_{\epsilon\rightarrow 0}\epsilon^{-1}\int_{\R^{n-1}\times \R}g(x,\epsilon^{-1}t)\, \rho(x,t)dx\, dt&=&\lim_{\epsilon\rightarrow 0}\int_{\R^{n-1}\times \R}g(x,t)\, \rho(x,\epsilon t)dx\, dt\\
	&=&\int_{\R^{n-1}\times \R}g(x,t)\, \rho(x,0)dx\, dt, 
\end{eqnarray*} \newline 
where the last identity is justified by the dominated convergence theorem. Let $\widetilde\lambda$ be the locally defined measure with density $\rho(x,0)$ with respect to $dx\, dt$. By construction, identity \eqref{eq:lambda_def} holds for this measure and all compactly supported $g$ as above. Since the left hand side of \eqref{eq:lambda_def} is independent of the trivializing chart, a standard partition of unity argument allows to glue the locally defined measures $\widetilde\lambda$ (as the chart varies over a cover) and define a global smooth positive measure $\lambda$ satisfying the required identity for all $g\in L^1_\mathrm{comp}(N\Sigma_1)$. Uniqueness follows from the fact that a measure is determined by the associated integral. Being a limit of rescalings (or because the local densities are independent of $t$), the measure $\lambda$ is dilation invariant. This proves the first half of the proposition. 

Let's turn to the second half. By another partition of unity argument we may assume that the functions $g_j$ and $g$ are all supported in the domain of a single chart trivializing the normal bundle. Using the same notation as before, we have to investigate the limiting behavior of \[
\epsilon_j^{-1}\int_{N\Sigma_1}g_j(x,\epsilon_j^{-1}t)\, \rho(x,t)dx\, dt=\int_{N\Sigma_1}g_j(x,t)\, \rho(x,\epsilon_jt)dx\, dt.
\]
Recall the following functional analysis lemma (a corollary of the uniform boundedness principle): if $\lim y_j=y$ in the weak topology of a Banach space $B$ and $\lim z_j=z$ in norm in the dual Banach space $B^*$, then $\lim \left\langle y_j, z_j\right\rangle= \left\langle y,z\right\rangle$ (where the angular bracket denotes the standard pairing of $B$ and $B^*$). The desired conclusion follows by applying the lemma to $B=L^1(K, dx\, dt)$, with $K$ the common compact support of the functions $g_j$, $y_j=g_j(x,t)$, $z_j=\rho(x, \epsilon_j t)$, $y=g(x,t)$, and $z=\rho(x,0)$. 
\end{proof}

Let now $X$ be a smooth vector field on $M$ that is tangent to $\Sigma_1$. Pushing it forward via the tubular neighborhood diffeomorphism \eqref{eq:tubular_neighborhood}, we obtain a vector field $\Phi_*X$ on $V$. By taking the first-order jet of $\Phi_*X$ in the fiber direction, we obtain a vector field $\mathfrak{X}$ on the whole normal bundle $N\Sigma_1$. This is explained in detail in Proposition \ref{prp:scaling_vectors} below. The idea is that $\mathfrak{X}$ is obtained as a scaling limit of $X$, in much the same way as the measure $\lambda$ was obtained from $\mu$. 

\begin{prp}[Scaling vector fields]\label{prp:scaling_vectors} Let $X$ be a smooth vector field on $M$ tangent to $\Sigma_1$. There exists a unique smooth vector field $\mathfrak{X}$ on $N\Sigma_1$ with the property that \begin{equation}\label{eq:mod_def}
		\mathfrak{X} g := \lim_{\epsilon\rightarrow 0} \mathrm{Dil}_{\epsilon^{-1}}\circ \Phi_*X\circ\mathrm{Dil}_{\epsilon} g\qquad \forall g\in C^\infty_c(N\Sigma_1), 
	\end{equation}  where the convergence is in the topology of test functions. The vector field $\mathfrak{X}$ commutes with the dilation operators $\mathrm{Dil}_\epsilon$. 

More generally, if $g_j$ converges weakly in $L^2_\mathrm{loc}(N\Sigma_1)$ to $g$ and $\epsilon_j\rightarrow 0$, then \begin{equation}\label{eq:mod_def_upgrade}
	\mathfrak{X} g=\lim_{j\rightarrow \infty} \mathrm{Dil}_{\epsilon_j^{-1}}\circ \Phi_*X\circ\mathrm{Dil}_{\epsilon_j} g_j
\end{equation}
in the weak $H^{-1}_\mathrm{loc}(N\Sigma_1)$ topology (that is, in the dual topology to $H^1_\mathrm{comp}(N\Sigma_1)$). \end{prp}

\begin{proof}
As in the proof of Proposition \ref{prp:scaling_measures}, we begin by working in local trivializing coordinates $(x,t)$ for the normal bundle $N\Sigma_1$. The tangency hypothesis on $X$ means that, in these coordinates, \begin{equation}\label{eq:tangent_local_coordinates}
	\Phi_*X = c(x,t)\cdot \nabla_x+d(x,t)\partial_t
\end{equation}
with \[
d(x,0)\equiv 0. 
\]
A simple computation shows that \begin{equation}\label{eq:conjugation_X}
	\mathrm{Dil}_{\epsilon^{-1}}\circ \Phi_*X\circ\mathrm{Dil}_{\epsilon} = c(x,\epsilon t)\cdot \nabla_x+\frac{d(x,\epsilon t)}{\epsilon}\partial_t, 
\end{equation}
as operators on $C^1$ functions. Since $\lim_{\epsilon\rightarrow 0}\frac{d(x,\epsilon t)}{\epsilon}=t\partial_td(x,0)$ in the $C^\infty$ topology, we define (locally)
 \begin{equation}\label{eq:mod_local_coordinates}
	\mathfrak{X} = c(x,0)\cdot \nabla_x+\partial_td(x,0) t\partial_t.
\end{equation}
Using a smooth partition of unity and the fact that vector fields are determined by their action on functions, one easily verifies that $\mathfrak{X}$ is a well-defined smooth vector field on $N\Sigma_1$, and that it is the unique vector field for which \eqref{eq:mod_def} holds. Being a limit of rescalings (or by formula \eqref{eq:mod_local_coordinates}), $\mathfrak{X}$ commutes with dilations. 

The proof of the upgraded convergence \eqref{eq:mod_def_upgrade} is similar to that of the corresponding limit \eqref{eq:lambda_def_upgrade} in Proposition \ref{prp:scaling_measures}. In fact, localizing to a chart and unraveling the definitions, the stated limit becomes \[
\lim_{j\rightarrow \infty} \int g_j\, Y_j^t\varphi\, dx\, dt = \int_{N\Sigma_1}g Y^t\varphi\, dx\, dt 
\]
where $\varphi\in H^1_\mathrm{comp}$, $Y_j=c(x,\epsilon_j t)\cdot \nabla_x+\frac{d(x,\epsilon_j t)}{\epsilon_j}\partial_t$, $Y=c(x,0)\cdot \nabla_x+\partial_td(x,0) t\partial_t$, and the superscript $t$ denotes formal adjunction with respect to Lebesgue measure $dx\, dt $. The validity of the last identity follows from the same functional analysis lemma used in the proof of Proposition \ref{prp:scaling_measures}, where now the Banach space $B$ is the $L^2$ space on a chart (say, with respect to $dx\, dt$). \end{proof}

\begin{rmk}\label{rmk:normal_operator} The proof of Proposition \ref{prp:scaling_vectors} shows that $\mathfrak{X}$ is formed from the expression of $\Phi_*X$ in local trivializing coordinates by freezing the $x$-coefficients at $t=0$ and taking the $1$-jet of the $t$-coefficient and freezing it at $t=0$. Notice that this construction is based on the choice of a tubular neighborhood diffeomorphism $\Phi$. The vector field $\mathfrak{X}$ itself is independent of such choice, if one makes the additional assumption that $d\Phi_{(p,0)}^{-1}(0,v)=v \mod T_p\Sigma_1$ (such normalization is always possible). We omit the straightforward check, as we will not explicitly need it. 
\end{rmk}

\subsection{Scaling lemma I: statement}\label{sec:scaling_statement}

Now that all preliminaries have been discussed, we can state the basic scaling lemma relating solutions of a sum-of-squares operator on a manifold to solutions of its "model" at a characteristic hypersurface. We start with the following data:
 \begin{itemize}
\item $M$, a smooth manifold.
\item $\mu$, a normalized smooth positive measure. 
\item $\mathcal{X}=\{X_1,\ldots, X_k\}$, a system of vector fields (as in Section \ref{sec:setup}, smooth and globally defined on $M$). 
\item $\Sigma$, a \emph{compact} embedded hypersurface, \emph{assumed to be characteristic and with no boundary grazing points} (with respect to $\mathcal{X}$). 
	\end{itemize}
We now take advantage of the theory of the last two sections. First of all, we may apply Proposition \ref{prp:flow_collar} to $\Sigma$, obtaining: \begin{itemize}
\item $\Sigma_0$, an embedded hypersurface with no boundary, extending $\Sigma$.
\item $Y$, a vector field tangent to $\Sigma_0$ and transversal to the boundary of $\Sigma$. 
	\end{itemize}
These are not unique, and a choice will be fixed throughout the rest of the section (if $\Sigma$ has no boundary, then $\Sigma_0=\Sigma$ and $Y=0$ are good choices). 

Denote by $\mathrm{Dil}_\epsilon$ the dilation operators acting on functions on $N\Sigma_0$, defined in Section \ref{sec:dilation_machinery}. Choose a tubular neighborhood diffeomorphism $\Phi:U\rightarrow V$, where $U$ is an open neighborhood of $\Sigma_0$, and $V$ is an open neighborhood of the zero section of $N\Sigma_0$. Notice that $N\Sigma$ may be viewed as a domain in $N\Sigma_0$, and the restriction of $\Phi$ to $\Phi^{-1}\left(V\cap N\overset{\circ}{\Sigma}\right)$ will serve as a tubular neighborhood diffeomorphism for $\overset{\circ}{\Sigma}$. This convention will always be assumed below. 

Next, Proposition \ref{prp:scaling_measures} applied to $\Sigma_0$ and $\mu$ gives: \begin{itemize}
\item $\lambda$, a	smooth positive measure on the normal bundle $N\Sigma_0$ uniquely defined by identity \eqref{eq:lambda_def}. 
	\end{itemize}
Since each $X_j$ is tangent to $\Sigma$ and $Y$ is tangent to the extended hypersurface $\Sigma_0$, Proposition \ref{prp:scaling_vectors} gives: \begin{itemize}
\item The vector fields $\mathfrak{X}_j=\lim_{\epsilon\rightarrow 0+} \mathrm{Dil}_{\epsilon^{-1}}\circ \Phi_*X_j\circ\mathrm{Dil}_{\epsilon}$ defined on $N\Sigma$.
\item The vector field $\mathfrak{Y}=\lim_{\epsilon\rightarrow 0+} \mathrm{Dil}_{\epsilon^{-1}}\circ \Phi_*Y\circ\mathrm{Dil}_{\epsilon}$ defined on $N\Sigma_0$. 
	\end{itemize}
Let \begin{equation}\label{eq:sum_of_squares}
P=\sum_j X_j^\dagger X_j
\end{equation}  be our basic sum-of-squares operator on the manifold $M$, where the $\dagger$ symbol denotes formal adjunction with respect to $\mu$. Finally, let \begin{equation}\label{eq:model_sum_of_squares}
\mathfrak{P}= \sum_j \mathfrak{X}_j^\dagger \mathfrak{X}_j
\end{equation} be the corresponding "model" sum-of-squares operator on the normal bundle of $\Sigma$, where this time the symbol $\dagger$ stands for formal adjunction with respect to $\lambda$.

\begin{lem}[Scaling Lemma]\label{lem:scaling} Let $M, \mathcal{X}, \mu, \Sigma$ be as above with $M$ compact and connected; in particular, $\Sigma$ is compact, characteristic and has no boundary grazing points (with respect to $\mathcal{X}$). Assume moreover that the following hold: \begin{enumerate}
\item The global $L^2$ Poincaré inequality is satisfied by $(M,\mathcal{X}, \mu)$ (cf.~Definition \ref{def:global_poincare}). 	
\item The operator $P$ (defined in \eqref{eq:sum_of_squares}) is hypoelliptic on $W\setminus \Sigma$, where $W$ is an open neighborhood of $\Sigma$. 
\item $P$ is exactly regular of order $k\in \mathbb{N}$ (cf.~Definition \ref{def:exact_reg}).
	\end{enumerate}
On $N\Sigma$, we have the measure $\lambda$ and the vector fields $\mathfrak{X}_j$ and $\mathfrak{Y}$ given by Propositions \ref{prp:scaling_measures} and \ref{prp:scaling_vectors}, as explained above. Let $f\in L^2_\mathrm{comp}(N\Sigma)$. Then the equation \begin{equation}\label{eq:model_equation}
\mathfrak{P} v=f
\end{equation} has a solution $v\in L^2(N\Sigma, \lambda)$ such that: \begin{itemize}
\item[(i)] $v$ has finite energy in the sense that $\mathfrak{X}_jv\in L^2(N\Sigma, \lambda)$ for all $j$. 
\item[(ii)] If $v$ is extended by zero to $N\Sigma_0$, then $\mathfrak{Y}v\in L^2(N\Sigma_0, \lambda)$.
\item[(iii)] if $f\in H^k_\mathrm{comp}(N\Sigma)$, then $v\in H^k_{\mathrm{vert}}(N\Sigma_0)$.
\end{itemize}
	\end{lem}

The rest of section \ref{sec:scaling} is occupied by the proof of the scaling lemma. 

\subsection{Scaling lemma II: construction of $v$}\label{sec:scaling2} For notation, refer to Section \ref{sec:scaling_statement}. Let $f\in L^2_\mathrm{comp}(N\Sigma)$ be the datum of the model equation \eqref{eq:model_equation}. The solution $v$ is found by the following process: \begin{enumerate}
	\item Rescale $f$, obtaining a function supported on an "$\epsilon$-neighborhood" of the zero section, and pull back the rescaled function to the manifold $M$ via the tubular neighborhood diffeomorphism.
    \item Solve the sum-of-squares equation \eqref{eq:sum_of_squares} on $M$ with datum the pulled-back function (corrected to have zero average).
    \item Cut off the solution near $\Sigma$, pull it back to the normal bundle, and undo the scaling. Call the resulting $\epsilon$-dependent function $v_\epsilon$. 
    \item Take a weak-$L^2$ limit of $v_\epsilon$ for $\epsilon$ tending to zero. 
	\end{enumerate}
Let us see the details. For $\epsilon$ small, $\mathrm{Dil}_\epsilon f$ is supported on the image $V$ of the tubular neighborhood diffeomorphism \eqref{eq:tubular_neighborhood}. Thus, we may define $f_\epsilon:=(\mathrm{Dil}_\epsilon f)\circ \Phi$, extended by zero outside $U$. By Proposition \ref{prp:scaling_measures},  \begin{equation}\label{eq:f_eps1}
	\lVert f_\epsilon\rVert_{L^2(M,\mu)}^2=\epsilon\lVert f\rVert_{L^2(N\Sigma,\lambda)}^2+o(\epsilon),
\end{equation} and \begin{equation}\label{eq:f_eps2}
\mathrm{av}_\mu(f_\epsilon)=\int_M f_\epsilon\, d\mu=\epsilon\int_{N\Sigma}f\, d\lambda+o(\epsilon).
\end{equation}
By assumption \emph{(1)} in Lemma \ref{lem:scaling}, we can apply Theorem \ref{thm:global_solutions} and obtain solutions $u_\epsilon$ of the equations
\begin{equation}\label{eq:equation_scaling}
	Pu_\epsilon= f_\epsilon-\mathrm{av}_\mu(f_\epsilon).
\end{equation}
The solutions may be taken to have zero average: $\mathrm{av}_\mu(u_\epsilon)=0$. By \eqref{eq:f_eps1} and the energy estimates \eqref{eq:energy_estimate} of Theorem \ref{thm:global_solutions}, we also have  \begin{equation}\label{eq:energy_estimates_1}
	\lVert u_\epsilon\rVert_{L^2(M, \mu)}^2=O(\epsilon\lVert f\rVert_{L^2(N\Sigma,\lambda)}^2)+o(\epsilon)
\end{equation} and \begin{equation}\label{eq:energy_estimates_2}
\sum_j \lVert X_ju_\epsilon\rVert_{L^2(M,\mu)}^2=O(\epsilon\lVert f\rVert_{L^2(N\Sigma,\lambda)}^2)+o(\epsilon)
\end{equation}
Fix a cut-off function $\chi\in C_c^\infty(U)$ identically equal to $1$ on a neighborhood of $\Sigma$, and set \[
v_\epsilon:=\mathrm{Dil}_{\epsilon^{-1}}\left((\chi u_\epsilon) \circ \Phi^{-1}\right),\]
where $(\chi u_\epsilon) \circ \Phi^{-1}$ is extended by zero outside $V$. Notice that $v_\epsilon:N\Sigma_0\rightarrow \mathbb{C}$ will not in general be supported on $N\Sigma\subset N\Sigma_0$, as control on the support was lost in solving the equation \eqref{eq:equation_scaling}. Moreover it is important, for later purposes, that the cut-off does not alter the solution in a neighborhood of $\Sigma$. By the dilation-invariance of $\lambda$ (identity \eqref{eq:dilation_invariance}) and the above estimates on $u_\epsilon$, \begin{eqnarray}
	\notag\lVert v_\epsilon\rVert_{L^2(N\Sigma_0,\lambda )}^2 &=& \epsilon^{-1}\lVert (\chi u_\epsilon) \circ \Phi^{-1} \rVert_{L^2(N\Sigma_0,\lambda)}^2\\
	\notag&=&O(\epsilon^{-1}\lVert u_\epsilon \rVert_{L^2(M,\mu)}^2)\\
	\label{eq:v_eps}&=&O(\lVert f\rVert_{L^2(N\Sigma_0,\lambda)}^2)+o(1).
\end{eqnarray}
In the second identity we used the fact that $\mu$ and the push-forward of $\lambda$ by $\Phi^{-1}$ have comparable densities on the support of $\chi$. The conclusion is that $v_\epsilon$ is uniformly bounded in $L^2(N\Sigma_0,\lambda)$. Thus, we can extract a weakly convergent subsequence: there exists \begin{equation}\label{eq:v}
	v\in L^2(N\Sigma_0,\lambda)
\end{equation} and $\epsilon_j\rightarrow 0$ such that $\lim v_{\epsilon_j}=v$ in the weak $L^2$ topology. By \eqref{eq:v_eps}, \begin{equation}\label{eq:v_bound}\lVert v\rVert_{L^2(N\Sigma_0, \lambda)}=O(\lVert f\rVert_{L^2(N\Sigma,\lambda)}).
\end{equation} 
This completes the construction of the function $v$ sketched at the beginning of this section. 

\subsection{Scaling lemma III: proof that $v$ is a solution} We now show that $v$ solves the model equation \eqref{eq:model_equation}. Let $u_\epsilon$ and $v_\epsilon$ be as in the previous section. Recall that equation \ref{eq:equation_scaling} is satisfied in the weak sense, that is, for each $\epsilon>0$ we have
\[
\sum_j\int_M X_j u_\epsilon\, \overline{X_jg}\, d\mu = \int_M(f_\epsilon-\mathrm{av}_\mu(f_\epsilon)) \overline{g}\, d\mu\qquad \forall g\in C^\infty(M).
\]
We are going to test the above identity on $g_\epsilon=(\mathrm{Dil}_\epsilon h)\circ \Phi$ (as usual, extended by zero outside $U$), for a fixed test function $h\in C^\infty_c(N\overset{\circ}{\Sigma})$, and look at the limit of both sides as $\epsilon$ tends to zero. By the definition of $f_\epsilon$ and \eqref{eq:f_eps2}, we get \[
\int_M(f_\epsilon-\mathrm{av}_\mu(f_\epsilon)) \overline{g_\epsilon}\, d\mu = \int_M \mathrm{Dil}_\epsilon(f\overline{h})\circ \Phi\, d\mu-\left(\epsilon\int_{N\Sigma}f\, d\lambda+o(\epsilon)\right)\int_M \mathrm{Dil}_\epsilon\overline{h}\circ \Phi\, d\mu. 
\]
By Proposition \ref{prp:scaling_measures}, we conclude that \begin{equation}\label{eq:equation_scaling_2}
	\lim_{\epsilon\rightarrow 0}\epsilon^{-1}\sum_j\int_M X_j u_\epsilon\, \overline{X_jg_\epsilon}\, d\mu =\int_{N\Sigma} f\, \overline{h}\, d\lambda.
\end{equation}
Notice how the correction needed to ensure zero average vanished in the limit. 

Recall that $\chi u_\epsilon =(\mathrm{Dil}_\epsilon v_\epsilon)\circ \Phi$. Fix $j=1,\ldots, k$ and compute  \begin{eqnarray*}
	X_j (\chi u_\epsilon) &=& X_j\left((\mathrm{Dil}_\epsilon v_\epsilon)\circ \Phi\right)\\
	&=&\left(\Phi_*X_j \circ \mathrm{Dil}_\epsilon v_\epsilon \right)\circ \Phi\\
	&=&\left(\mathrm{Dil}_\epsilon w_\epsilon \right)\circ \Phi, 
\end{eqnarray*} where all identities are on $U$ and \[
w_\epsilon:=\mathrm{Dil}_{\epsilon^{-1}}\circ \Phi_*X_j \circ \mathrm{Dil}_\epsilon v_\epsilon.
\]
Since $v_{\epsilon_j}$ converges weakly in $L^2$ to $v$, by the second half of Proposition \ref{prp:scaling_vectors} \begin{equation}\label{eq:lim_w_epsilon}
\lim w_{\epsilon_j}=\mathfrak{X}_j v
\end{equation} in the sense of distributions on $N\overset{\circ}{\Sigma}$. Observe that, since $X_j$ need not be tangent to $\Sigma_0$ (outside $\Sigma$), the limit may not exist on $N\Sigma_0$. The energy estimates \eqref{eq:energy_estimates_1} and \eqref{eq:energy_estimates_2} give \begin{eqnarray}
\notag\lVert X_j (\chi u_\epsilon)\rVert_{L^2(M, \mu)}^2&=&O\left(\lVert X_j u_\epsilon\rVert_{L^2(M, \mu)}^2+\lVert u_\epsilon\rVert_{L^2(M,\mu)}^2\right)\\
\label{eq:X_j_bound}&=& O\left(\epsilon \lVert f\rVert_{L^2(N\Sigma, \lambda)}^2\right)+o(\epsilon).
\end{eqnarray}
By dilation-invariance of $\lambda$ and the local comparability of the densities of $\lambda$ and $\Phi_*\mu$, we find
\begin{equation}\label{eq:w_epsilon_bound}
\epsilon \lVert  w_\epsilon\rVert_{L^2(N\Sigma_0, \lambda)}^2=\lVert \mathrm{Dil}_\epsilon w_\epsilon\rVert_{L^2(N\Sigma_0,\lambda)}^2=O\left(\lVert X_j (\chi u_\epsilon) \rVert_{L^2(M,\mu)}^2 \right).
\end{equation}
Combining \eqref{eq:X_j_bound} and \eqref{eq:w_epsilon_bound} we obtain that $w_\epsilon$ is uniformly bounded in $L^2(N\Sigma, \lambda)$, which in turn shows that \begin{equation}\label{eq:Xv}
	\mathfrak{X}_j v\in L^2(N\Sigma, \lambda)\end{equation} and that \eqref{eq:lim_w_epsilon} is also a weak $L^2$ limit. Now, applying the same argument with the outward pointing vector field $Y$ in place of $X_j$ shows that \begin{equation}\label{eq:Yv}\mathfrak{Y}v\in L^2(N\Sigma_0, \lambda),\end{equation} 
because, crucially, $Y$ is tangent to the whole $\Sigma_0$. 

We can treat the term $X_jg_\epsilon$ similarly to $X_j(\chi u_\epsilon)$. We have \[
X_j g_\epsilon = \left(\mathrm{Dil}_\epsilon z_\epsilon \right)\circ \Phi, 
\] where 
\[
z_\epsilon:=\mathrm{Dil}_{\epsilon^{-1}}\circ \Phi_*X_j \circ \mathrm{Dil}_\epsilon h, 
\]
which is a test function supported on $\mathrm{supp}(h)\subseteq N\overset{\circ}{\Sigma}$. By Proposition \ref{prp:scaling_vectors}, $z_\epsilon$ converges to $\mathfrak{X}_jh$ in $L^2(N\Sigma, \lambda)$ (in fact, in the topology of test functions) as $\epsilon$ tends to $0$. On the support of $g_\epsilon$ the cut-off $\chi$ is identically one ($\epsilon$ is small), so \[
X_j u_\epsilon\,\overline{X_j g_\epsilon}=\left(\mathrm{Dil}_\epsilon w_\epsilon \overline{z_\epsilon} \right)\circ \Phi.
\]
We have seen above that $w_{\epsilon_j}$ converges weakly in $L^2$ to $\mathfrak{X}_j v$, $z_{\epsilon_j}$ converges (strongly) in $L^2$ to $\mathfrak{X}_jh$, and $z_\epsilon$ vanishes outside $\mathrm{supp}(h)$; it follows that the product $w_{\epsilon_j}\overline{z_{\epsilon_j}}$ converges weakly in $L^1_\mathrm{comp}$ to $\mathfrak{X}_j v\, \overline{\mathfrak{X}_jh}$. By the second half of Proposition \ref{prp:scaling_measures}, \begin{eqnarray*}
&&\lim_{m\rightarrow \infty}\epsilon_m^{-1}\int_M X_j u_{\epsilon_m}\,\overline{X_j g_{\epsilon_m}}\, d\mu\\
&=&\lim_{m\rightarrow \infty}\epsilon_m^{-1}\int_M \left(\mathrm{Dil}_\epsilon w_{\epsilon_m} \overline{z_{\epsilon_m}} \right)\circ \Phi\, d\mu\\
&=& \int_{N\Sigma} \mathfrak{X}_j v\, \overline{\mathfrak{X}_jh}\, d\lambda. 
\end{eqnarray*}
Combining this with \eqref{eq:equation_scaling_2}, we conclude that \begin{equation}
\sum_j\int_{N\Sigma} \mathfrak{X}_j v\, \overline{\mathfrak{X}_jh}\, d\lambda=\int_{N\Sigma} f\, \overline{h}\, d\lambda\qquad \forall h\in C^\infty_c(N\Sigma), 
\end{equation}
that is, $\mathfrak{P} v=f$ on $N\Sigma$ in the sense of distributions. 

\subsection{Scaling lemma IV: Dirichlet boundary conditions} In the previous section, we established conclusion \emph{(i)} of the scaling lemma. In this section, we prove \emph{(ii)} by showing that $v$ vanishes on $N\Sigma_0\setminus N\Sigma$ (recall that $v$ is an element of $L^2(N\Sigma_0, \lambda)$). This is enough, in view of \eqref{eq:Yv}. To achieve this conclusion, we take advantage of the hypoellipticity assumption \emph{(2)} in Lemma \ref{lem:scaling}. More precisely, we need the following consequence of hypoellipticity. Let $W$ be the subset with the property that $P=\sum_j X_j^\dagger X_j$ is hypoelliptic on $W\setminus \Sigma$. 

\begin{prp}[Closed graph bound]\label{prp:closed_graph}
	Let $\Omega\subseteq W\setminus \Sigma$ be open. Let $K$ be a compact subset of $\Omega$ that is the closure of its interior. Then there exists a constant $C_K$, a compact set $K'\subseteq \Omega$, and an integer $\ell\in \mathbb{N}$ such that the estimate \begin{equation}\label{eq:hypoellipticity_bound}
		\lVert u\rVert_{C(K)}\leq C_K \left(\lVert u\rVert_{\mathcal{D}_\mathcal{X}}+\lVert Pu\rVert_{C^\ell(K')}\right)
	\end{equation}
holds for any $u\in \mathcal{D}_\mathcal{X}$ such that $Pu\in C^\infty(\Omega)$.  
\end{prp}

\begin{proof}	
	Let \[
	\mathcal{F}=\{f\in \mathcal{D}_\mathcal{X}\colon\, Pf\in C^\infty(\Omega)\}, 
	\]
	which is a Fréchet space with respect to the family of seminorms $\lVert f\rVert_{\mathcal{D}_\mathcal{X}}+\lVert Pf\rVert_{C^\ell(L)}$, as $L$ varies over the collection of compact subsets of $\Omega$ and $\ell\in \mathbb{N}$ (the $C^\ell$ norm here may be defined using finitely many local charts). Let $K$ be the set in the statement. Since $P$ is hypoelliptic on $\Omega$, the restriction operator \[
	\mathcal{F} \ni f\longmapsto f_{|K},
	\]
	takes values in the Banach space $C^0(K)$. Since $K$ is the closure of its interior, it is a closed operator. The conclusion now follows from the closed graph theorem for Fréchet spaces. \end{proof}

We now prove statement \emph{(ii)}. As all we care about is what happens to $v$ in a small outer collar of $N\Sigma$, we may assume that $U$, the domain of the tubular neighborhood diffeomorphism of $\Sigma_0$, is contained in $W$. This makes the argument below less cumbersome.

Since $f\in L^2_\mathrm{comp}(N\Sigma)$, $f_\epsilon=(\mathrm{Dil}_\epsilon f)\circ \Phi$ vanishes outside $\Phi^{-1}(N\Sigma)$. Hence, \[
Pu_\epsilon =f_\epsilon-\mathrm{av}_\mu(f_\epsilon)
\] is constantly equal to $-\mathrm{av}_\mu(f_\epsilon)$, and in particular smooth, on the nonempty open set \[\Omega:=U\setminus \Phi^{-1}(N\Sigma)\subseteq W.\] Let $K\subseteq \Omega$ be a compact set that is the closure of its interior. By Proposition \ref{prp:closed_graph}, there exist $C_K$, $\ell$, $K'\subseteq \Omega$ such that \[
\lVert u\rVert_{C(K)}\leq C_K \left(\lVert u\rVert_{\mathcal{D}_\mathcal{X}}+\lVert Pu\rVert_{C^\ell(K')}\right)
\]
for every $u\in \mathcal{D}_\mathcal{X}$ such that $Pu\in C^\infty(\Omega)$. Applying this bound to $u_\epsilon$ and recalling the bounds \eqref{eq:f_eps2}, \eqref{eq:energy_estimates_1} and \eqref{eq:energy_estimates_2}, we get\[
\lVert u_\epsilon\rVert_{C(K)}=O(\epsilon^\frac{1}{2}), 
\]
where the constant implicit in the notation depends on $K$, but not on $\epsilon$. Recalling that $\mathrm{Dil}_{\epsilon}v_\epsilon=(\chi u_\epsilon)\circ \Phi^{-1}$, we conclude that  \begin{equation}\label{eq:v_epsilon_small}
\max_{(p,v)\in \Phi(K)}|v_\epsilon (p,\epsilon^{-1}v)|=O(\epsilon^\frac{1}{2}).
\end{equation}
If $K$ is taken to be a small neighborhood of a point $p\in \Sigma_0\setminus \Sigma$ (as we may, by our assumption that $U\subseteq W$), then $\Phi(K)$ is a neighborhood of $(p,0)$ in $N\Sigma_0$. Taking the limit as $\epsilon$ tends to zero of \eqref{eq:v_epsilon_small} shows that $v=\lim v_{\epsilon_j}$ vanishes almost everywhere on a set of the form $\{(p',v)\in N\Sigma_0\colon\, p'\in B\}$, where $B$ is a small neighborhood of $p$ in $\Sigma_0\setminus \Sigma$. Countably many sets of this form cover $N\Sigma_0\setminus N\Sigma$. The proof of conclusion \emph{(ii)} of Lemma \ref{lem:scaling} is complete.

\subsection{Scaling lemma V: vertical differentiability} We are left with the proof of statement \emph{(iii)} in the scaling lemma.
As in the previous section, we use a bound that follows from the exact regularity assumption by an application of the closed-graph theorem. Namely, by hypothesis \emph{(3)} of Lemma \ref{lem:scaling} one may derive the existence of a constant $C$ such that  \begin{equation}\label{eq:closed_graph}
	\lVert u\rVert_{H^k(M)}\leq C\lVert f\rVert_{H^k(M)}\end{equation}
whenever $u\in \mathcal{D}_\mathcal{X}$ is the zero average solution of $Pu=f$ (this is the first remark at the end of Section \ref{sec:exact_regularity}). 

We now work under the assumption that $f\in H^k_\mathrm{comp}(N\Sigma)$. We start by estimating the Sobolev norms of $\mathrm{Dil}_\epsilon f$. Let $(x,t)=(x_1,\ldots, x_{n-1}, t)$ be local trivializing coordinates on $N\Sigma_0$. If $\alpha=(\alpha_1, \ldots, \alpha_n)$ is a multi-index of length $\leq k$, then \[
\frac{\partial^{\alpha_1}}{\partial x_1^{\alpha_1}}\cdots\frac{\partial^{\alpha_{n-1}}}{\partial x_{n-1}^{\alpha_{n-1}}} \frac{\partial^{\alpha_n}}{\partial t^{\alpha_n}}\mathrm{Dil}_\epsilon f =\epsilon^{-\alpha_n}\mathrm{Dil}_\epsilon \left(\frac{\partial^{\alpha_1}}{\partial x_1^{\alpha_1}}\cdots\frac{\partial^{\alpha_{n-1}}}{\partial x_{n-1}^{\alpha_{n-1}}} \frac{\partial^{\alpha_n}}{\partial t^{\alpha_n}} f\right).
\]
Thus, a change of variables gives \[
\int_{K\times \R} \left|\frac{\partial^{\alpha_1}}{\partial x_1^{\alpha_1}}\cdots\frac{\partial^{\alpha_{n-1}}}{\partial x_{n-1}^{\alpha_{n-1}}} \frac{\partial^{\alpha_n}}{\partial t^{\alpha_n}}\mathrm{Dil}_\epsilon f\right|^2\, dx\, dt=O_K(\epsilon^{1-2k})
\]
for any compact set $K$ contained in the domain of the coordinates $(x_1,\ldots, x_{n-1})$. The functions $f_\epsilon=(\mathrm{Dil}_\epsilon f)\circ \Phi$ have a common compact support contained in $U$, so their $H^k(M)$ norms can be computed using finitely many systems of local trivializing coordinates for the normal bundle (via the tubular neighborhood diffeomorphism $\Phi$). Thus,  
 \[\lVert f_\epsilon\rVert_{H^k(M)}^2=O\left(\epsilon^{1-2k}\right).\]
The function $u_\epsilon$ is the zero average solution of $Pu_\epsilon=f_\epsilon-\mathrm{av}_\mu(f_\epsilon)$. By \eqref{eq:f_eps2} and \eqref{eq:closed_graph},  \begin{equation}\label{eq:sobolev_norm}
\lVert u_\epsilon \rVert_{H^k(M)}^2=O\left(\lVert f_\epsilon \rVert_{H^k(M)}^2\right)+O(\epsilon^2) =O(\epsilon^{1-2k}).	\end{equation}
Recall that $v_\epsilon=\mathrm{Dil}_{\epsilon^{-1}}\left((\chi u_\epsilon)\circ \Phi^{-1}\right)$. If $(x,t)$ are local trivializing coordinates on $N\Sigma_0$, then \[
\frac{\partial^kv_\epsilon}{\partial t^k}=\epsilon^k\mathrm{Dil}_{\epsilon^{-1}}\frac{\partial^k}{\partial t^k}\left((\chi u_\epsilon)\circ \Phi^{-1}\right).
\] For a compact set $K$ contained in the domain of the coordinates $(x_1,\ldots, x_{n-1})$, \[
\int_{K\times \R} \left|\frac{\partial^kv_\epsilon}{\partial t^k}\right|^2\, dx\, dt=\epsilon^{2k-1}\int_{K\times \R} \left|\frac{\partial^k}{\partial t^k}\left((\chi u_\epsilon)\circ \Phi^{-1}\right)\right|^2\, dx\, dt.
\] 
Since the integral on the right hand side is $O_K(\lVert u_\epsilon \rVert_{H^k(M)}^2)$, \eqref{eq:sobolev_norm} gives \[
\int_{K\times \R} \left|\frac{\partial^kv_\epsilon}{\partial t^k}\right|^2\, dx\, dt=O_K(1)
\]
uniformly in $\epsilon$. Since $v_\epsilon$ converges weakly to the solution $v$ (modulo passage to a subsequence), this yields \[
\int_{K\times \R} \left|\frac{\partial^kv}{\partial t^k}\right|^2\, dx\, dt<+\infty. 
\]
A standard interpolation argument now gives the $L^2$ integrability of all vertical derivatives of order $\leq k$ (alternatively, one may use the second remark at the end of Section \ref{sec:exact_regularity}, which allows us to directly argue as above for any $j\leq k$). This completes the proof that $v\in H^k_\mathrm{vert}(N\Sigma_0)$.

\section{Failure of exact regularity of worm sublaplacians}\label{sec:worm}

\subsection{Diederich--Fornaess worm domains}\label{sec:DF}
We recall the definition of the classical Diederich--Fornaess worm domain, referring to \cite{DiederichFornaess1977} for more details. Fix: \begin{itemize}
\item A compact interval of positive length $I\subseteq \R$.
\item A \emph{non-zero} real number $\alpha$. 
\item A smooth function $\eta:\R\rightarrow [0,\infty)$ such that $I=\{x\in \R\colon\, \eta(x)=0\}$. 
	\end{itemize} 
Denote by $(z_1,z_2)$ the standard holomorphic coordinates on $\mathbb{C}^2$, and let $\Omega\subseteq \mathbb{C}^2$ be the open set defined by the conditions  \begin{equation}\label{eq:defining_function}
z_1\neq 0, \qquad \rho(z_1,z_2):=|z_2-e^{i\alpha \log |z_1|^2}|^2-1+\eta(\log |z_1|^2)<0.
\end{equation}
The interval $I$ and $\alpha$ should be viewed as the geometric parameters of the domain, while $\eta$ is an auxiliary function controlling how the ends of a preliminary unbounded domain (corresponding to $\eta\equiv 0$) are "capped" in order to make it smoothly bounded. It is shown in \cite{DiederichFornaess1977} that, if $\eta$ satisfies appropriate monotonicity and convexity assumptions, then:  \begin{itemize}
	\item $\Omega$ is a bounded pseudoconvex domain with smooth and connected boundary $M$.
	\item The real hypersurface $M=\partial \Omega$, which contains the compact annulus \begin{equation}\label{eq:annulus}
	\Sigma=\{(z_1,0)\colon\, \log |z_1|^2\in I\}, 
	\end{equation}
 is strongly pseudoconvex at every point of $M\setminus \Sigma$. 
\end{itemize} 
As the precise conditions on $\eta$ play no role in what follows, we do not comment further on them. We call any domain $\Omega$ with the above properties a \emph{D--F worm domain}. 

\subsection{Computing a generator of the CR bundle} Let $\Omega$ be a D--F worm domain with defining function $\rho$ as in \eqref{eq:defining_function}. As explained in Section \ref{sec:scv}, its boundary $M$ has a natural CR structure. Here it is convenient to work with the complex rank one bundle \[
T^{1,0}M:=T^{1,0}\mathbb{C}^2\cap \mathbb{C}TM, 
\]
equivalently, the complexification of the real distribution $HM$ defined in \eqref{eq:horizontal_bundle} is $T^{1,0}M\oplus \overline{T^{1,0}M}$. Crucially, the fiber of $HM$ at each point $p$ of the annulus $\Sigma$ coincides with $T_p\Sigma$; equivalently $T^{1,0}_pM=T_p^{1,0}\Sigma$. This is an immediate consequence of the complex-analyticity of the annulus. 

\begin{prp}\label{prp:generator}
In an open neighborhood $W$ of $\Sigma$ in $M$ one can find local coordinates $(z,t)\in \mathbb{C}\times \R$ such that:\begin{enumerate}
	\item $\Sigma=W\cap \{\log |z|^2\in I, \quad t=0\}$
\item The bundle $T^{1,0}M$ has a global generator $Z$ that admits the following coordinate representation in $W$: \[
Z=(1+O(d_\Sigma))\frac{\partial}{\partial z} +(i\alpha z^{-1}	t+O(d_\Sigma^2))\frac{\partial}{\partial t}.
\]
Here $d_\Sigma$ denotes the distance to $\Sigma$ (with respect to any Riemannian metric, the precise choice does not matter).
\end{enumerate}
	\end{prp}

\begin{proof}
The type $(1,0)$ vector field \[
\widetilde{Z}=\frac{\partial\rho}{\partial z_2}\, \frac{\partial}{\partial z_1}-\frac{\partial\rho}{\partial z_1}\, \frac{\partial}{\partial z_2}
\]
is nowhere vanishing on $M$. Being tangent to $M$, it restricts to a global generator of $T^{1,0}M$. Simple computations give the formulas \begin{eqnarray}\label{eq:def_worm}
\label{eq:rho}	\rho&=&|z_2|^2-2\mathrm{Re}(z_2e^{-i\alpha \log |z_1|^2})+\eta(\log |z_1|^2) \\
\label{eq:partial_rho_1}	\frac{\partial \rho}{\partial z_1}&=& -2\alpha z_1^{-1}\mathrm{Im}(z_2e^{-i\alpha \log |z_1|^2})+z_1^{-1}\eta'(\log |z_1|^2)\\
\label{eq:partial_rho_2}	 \frac{\partial \rho}{\partial z_2} &=& \overline{z_2}-e^{-i\alpha \log|z_1|^2}.
\end{eqnarray} 
Hence \begin{equation}\label{eq:Z_tilde}
	\widetilde{Z}=(\overline{z_2}-e^{-i\alpha \log |z_1|^2})\frac{\partial}{\partial z_1}+\left(2\alpha z_1^{-1}\mathrm{Im}(z_2e^{-i\alpha\log|z_1|^2})-z_1^{-1}\eta'(\log|z_1|^2)\right)\frac{\partial}{\partial z_2}.
\end{equation}
On a neighborhood of $\Sigma$ in the ambient space $\mathbb{C}^2$, we may use the nonholomorphic coordinates \[(z=z_1,t=\mathrm{Im}(z_2e^{-i\alpha\log |z_1|^2}),\rho)\in \mathbb{C}\times \R^2.\] Using \eqref{eq:partial_rho_1}, \eqref{eq:partial_rho_2}, and \begin{eqnarray*}
	&&\frac{\partial z}{\partial z_1}=1,\quad \frac{\partial \overline{z}}{\partial z_1}=0,\quad \frac{\partial t}{\partial z_1}=-\alpha z_1^{-1}\mathrm{Re}(z_2e^{-i\alpha \log |z_1|^2})\\
	&&\frac{\partial z}{\partial z_2}=0,\quad \frac{\partial \overline{z}}{\partial z_2}=0,\quad \frac{\partial t}{\partial z_2}=-\frac{ie^{-i\alpha \log |z_1|^2}}{2},
\end{eqnarray*} 
we find \begin{eqnarray*}
	\frac{\partial}{\partial z_1} &=& \frac{\partial}{\partial z} -\alpha z_1^{-1}\mathrm{Re}(z_2e^{-i\alpha \log |z_1|^2})\frac{\partial}{\partial t}+(-2\alpha z_1^{-1}\mathrm{Im}(z_2e^{-i\alpha \log |z_1|^2})+z_1^{-1}\eta'(\log |z_1|^2))	\frac{\partial}{\partial \rho}\\
	&=& \frac{\partial}{\partial z} -\alpha z^{-1}\mathrm{Re}(z_2e^{-i\alpha \log |z|^2})\frac{\partial}{\partial t}+(-2\alpha z^{-1}t+z^{-1}\eta'(\log |z|^2))	\frac{\partial}{\partial \rho}\\
	\frac{\partial}{\partial z_2}&=& -\frac{ie^{-i\alpha \log|z|^2}}{2}	\frac{\partial}{\partial t}+(\overline{z_2}-e^{-i\alpha \log|z|^2})	\frac{\partial}{\partial \rho}.
\end{eqnarray*} 
In the formula for $\frac{\partial}{\partial z_1}$, we are thinking of $z_2$ as a function of the local coordinates $z, t, \rho$. Plugging these identities into \eqref{eq:Z_tilde}, we get 
\begin{eqnarray*}
	&&\widetilde{Z}=(\overline{z_2}-e^{-i\alpha \log |z|^2})\left(\frac{\partial}{\partial z} -\alpha z^{-1}\mathrm{Re}(z_2e^{-i\alpha \log |z|^2})\frac{\partial}{\partial t}\right)\\
	&&+\left(2\alpha z^{-1}t-z^{-1}\eta'(\log|z|^2)\right)\left(-\frac{ie^{-i\alpha \log|z|^2}}{2}	\frac{\partial}{\partial t}\right).
\end{eqnarray*}
Recall that $\widetilde{Z}\rho=0$, and accordingly the coefficient of $\frac{\partial}{\partial \rho}$ vanishes. 

Let $W$ be a neighborhood of $\Sigma$ \emph{in the boundary $M$}, where $(z,t)$ is a system of local coordinates. Then conclusion \emph{(1)} of the proposition holds. On $W$ we have  \begin{eqnarray*}
	&&z,z^{-1}=O(1), \\
	&&z_2=O(d_\Sigma), \quad \eta'(\log|z|^2)=O(d_\Sigma^N)\quad \forall N, \\
	&& \mathrm{Re}(z_2e^{-i\alpha \log |z|^2})=\frac{|z_2|^2+\eta(\log |z|^2)}{2}=O(d_\Sigma^2), 
\end{eqnarray*}
where we used \eqref{eq:rho} for the last one. Inserting this information into our formula for $\widetilde{Z}$ we find \[
\widetilde{Z}=(-e^{-i\alpha \log |z|^2}+O(d_\Sigma))\frac{\partial}{\partial z} +(-i\alpha z^{-1}e^{-i\alpha \log|z|^2}t+O(d_\Sigma^2))\frac{\partial}{\partial t}. 
\]
The desired generator of $T^{1,0}M$ is \[
Z:=-e^{i\alpha \log |z|^2}\widetilde{Z}=(1+O(d_\Sigma))\frac{\partial}{\partial z} +(i\alpha z^{-1}	t+O(d_\Sigma^2))\frac{\partial}{\partial t}. 
\]
\end{proof}

\subsection{System of vector fields and measure}\label{sec:worm_vector_measure}
Let $\Omega$ be a D--F worm domain, $M=\partial \Omega$, and $Z$ the global generator of $T^{1,0}M$ of Proposition \ref{prp:generator}. On $M$, we define the system of vector fields \[
\mathcal{X}=\{X_1,X_2\}:=\{\mathrm{Re}\,Z, -\mathrm{Im}\, Z\}, 
\]
so that $Z=X_1-iX_2$. Let $(z, t)$ be local coordinates on a neighborhood $W$ of $\Sigma$ as in Proposition \ref{prp:generator}. Write $z=x_1+ix_2$ with $(x_1,x_2)\in \R^2$. Then \begin{eqnarray}
\label{eq:X1}	X_1&=&\left(\frac{1}{2}+O(d_\Sigma)\right)\frac{\partial}{\partial x_1}+O(d_\Sigma)\frac{\partial}{\partial x_2}+\left(\alpha \frac{x_2}{x_1^2+x_2^2}t+O(d_\Sigma^2)\right)\frac{\partial}{\partial t}\\
\label{eq:X2}	X_2&=&\left(\frac{1}{2}+O(d_\Sigma)\right)\frac{\partial}{\partial x_2}+O(d_\Sigma)\frac{\partial}{\partial x_1}+\left(-\alpha \frac{x_1}{x_1^2+x_2^2}t+O(d_\Sigma^2)\right)\frac{\partial}{\partial t} 
\end{eqnarray}
We highlight a few key properties of the vector fields $X_1,X_2$: \begin{itemize}
\item[(V1)] $X_1$ and $X_2$ generate the distribution $HM$.
\item[(V2)] $X_1$ and $X_2$ are tangent to $\Sigma$.
\item[(V3)] $X_1$, $X_2$ and their commutator $[X_1,X_2]$ span the tangent bundle $TM$ at every point of $M\setminus \Sigma$.
	\end{itemize}
The first property is clear by construction; the second follows from the complex-analyticity of $\Sigma$, or the formulas above; the third is a consequence of the strong pseudoconvexity of $M$ outside the annulus (recalled in Section \ref{sec:DF}), or a direct computation. 

 In order to define a sublaplacian on $M$ as in Section \ref{sec:global_solutions}, we need to choose a smooth positive normalized measure $\mu$. We impose a constraint on $\mu$, that simplifies the analysis of the model operator below:  \begin{itemize}
\item[(M)] in the local coordinates $(x_1,x_2,t)$ defined above we have $d\mu= \rho(x_1,x_2,t)\, dx_1\, dx_2\, dt$, with \[\rho(x_1,x_2,0)\equiv 1\] for all $(x_1,x_2)$ such that $\log (x_1^2+x_2^2)\in I$.
\end{itemize}
Given any such measure, let \begin{equation}\label{eq:worm_sublaplacian}
P=X_1^\dagger X_1+X_2^\dagger X_2
\end{equation} be the associated sum-of-squares operator, which we call a \emph{worm sublaplacian}. By property (V1), we may apply Theorem \ref{thm:global_solutions_cr}, which gives existence and uniqueness of global finite energy solutions for the equation $Pu=f$ with square-integrable data, as explained in Sections \ref{sec:global_solutions} and \ref{sec:scv}. Moreover, by property (V3) and a celebrated hypoellipticity theorem of Hörmander, the operator $P$ is hypoelliptic on $M\setminus \Sigma$. 

\subsection{Model operator} We now compute the model operator $\mathfrak{P}$, specializing to the present setting the constructions of Section \ref{sec:scaling_statement}. We begin with some preliminary geometric observations. 

The annulus $\Sigma$ is a compact surface, bounded by two disjoint circles. By properties (V1) and (V2) of Section \ref{sec:worm_vector_measure}, $\Sigma$ is characteristic and has no grazing boundary points. As in Section \ref{sec:scaling_statement}, in this situation we may extend $\Sigma$ to a slightly larger surface $\Sigma_0$, which may be taken to be an open annulus, and we have a vector field $Y$ tangent to $\Sigma_0$ and transversal to the two boundary circles of $\Sigma$. We remark that the part of $\Sigma_0$ lying outside $\Sigma$ cannot be complex-analytic, because $M\setminus \Sigma$ is strongly pseudoconvex. A convenient choice of $Y$ is \begin{eqnarray}
\notag Y&=& 2x_1X_1+2x_2X_2\\
&=&	\left(x_1\frac{\partial}{\partial x_1}+x_2\frac{\partial}{\partial x_2}\right)+O(d_\Sigma)\frac{\partial}{\partial x_1}+O(d_\Sigma)\frac{\partial}{\partial x_2}+O(d_\Sigma^2)\frac{\partial}{\partial t}
	\end{eqnarray}

The normal bundle of $\Sigma_0$ in $M$ is globally trivial. We identify it with the product manifold $\Sigma_0\times \R$, and choose a tubular neighborhood diffeomorphism \[
\Phi: U\rightarrow V\subseteq \Sigma_0\times \R
\]
with $U\subseteq W$ and \[\Phi=(x_1, x_2,t)\qquad \text{on }U\cap \{\log(x_1^2+x_2^2)\in I\}.\] To achieve this, notice that $(x_1,x_2)$ restricts to global coordinates on $\Sigma_0$ (for this, one may need to shrink $\Sigma_0$, as we may without impacting the rest of the argument). If $T$ is a vector field transversal to $\Sigma_0$ and identical to $\frac{\partial}{\partial t}$ on $U\cap \{\log(x_1^2+x_2^2)\in I\}$, then for each $p$ in a sufficiently small neighborhood of $\Sigma_0$ there exist unique $(x_1,x_2,t)$ such that $p$ is reached starting at the point of $\Sigma_0$ of coordinates $(x_1,x_2)$ and flowing along $T$ for a time $t$ . 

We now compute the "model vector fields" $\mathfrak{X}_1, \mathfrak{X}_2$, and $\mathfrak{Y}$ given by Proposition \ref{prp:scaling_vectors}. Using the identities \eqref{eq:X1} and \eqref{eq:X2} of the previous section and recalling Remark \ref{rmk:normal_operator}, we obtain the two vector fields, defined on $N\Sigma\equiv \Sigma\times \R$,
 \begin{equation}\label{eq:X_model}
\mathfrak{X}_1=\frac{1}{2}\frac{\partial}{\partial x_1}+\alpha \frac{x_2}{x_1^2+x_2^2}t\frac{\partial}{\partial t}, \qquad \mathfrak{X}_2=\frac{1}{2}\frac{\partial}{\partial x_2}-\alpha \frac{x_1}{x_1^2+x_2^2}t\frac{\partial}{\partial t},\end{equation}
and the vector field, defined on $N\Sigma_0$, \begin{equation*}
\mathfrak{Y}=x_1\frac{\partial}{\partial x_1}+x_2\frac{\partial}{\partial x_2}+\mathfrak{Y}_0, 
\end{equation*}
where $\mathfrak{Y}_0$ vanishes on $N\Sigma$. Recall that we only let $\mathfrak{Y}$ act on functions $v$ that vanish off $N\Sigma$ (conclusion \emph{(ii)} of Lemma \ref{lem:scaling}), and on such functions $\mathfrak{Y}_0v=0$. Thus, for our purposes we may remove the term $\mathfrak{Y}_0$ and redefine \begin{equation}\label{eq:Y_model}
\mathfrak{Y}=x_1\frac{\partial}{\partial x_1}+x_2\frac{\partial}{\partial x_2}. 
\end{equation}

From the assumption (M) on the measure $\mu$ (cf.~Section \ref{sec:worm_vector_measure}), one immediately sees that the measure $\lambda$ on the normal bundle given by Proposition \ref{prp:scaling_measures} is Lebesgue measure in the $(x_1,x_2,t)$-coordinates. Thus,  \[
\mathfrak{X}_1^\dagger=-\frac{1}{2}\frac{\partial}{\partial x_1}-\alpha \frac{x_2}{x_1^2+x_2^2}\left(t\frac{\partial}{\partial t}+1\right), \qquad \mathfrak{X}_2^\dagger=-\frac{1}{2}\frac{\partial}{\partial x_2}+\alpha \frac{x_1}{x_1^2+x_2^2}\left(t\frac{\partial}{\partial t}+1\right).
\]
A computation gives the following formula for the model operator \eqref{eq:model_sum_of_squares}:
\begin{equation}\label{eq:model_equation_formula}
	\mathfrak{P}=-\frac{1}{4}\Delta_{x_1,x_2}+\frac{\alpha}{2}\left(- \frac{x_2}{x_1^2+x_2^2}\frac{\partial}{\partial x_1}+ \frac{x_1}{x_1^2+x_2^2}\frac{\partial}{\partial x_2}\right)\left(2t\frac{\partial}{\partial t}+1\right)
	-\frac{\alpha^2}{x_1^2+x_2^2}\left(t\frac{\partial}{\partial t}+1\right)t\frac{\partial}{\partial t}.
\end{equation}

\subsection{Model Dirichlet problem}\label{sec:gluing}

Motivated by the scaling lemma, we now study the Dirichlet problem for the operator $\mathfrak{P}=\mathfrak{X}_1^\dagger \mathfrak{X}_1+\mathfrak{X}_2^\dagger \mathfrak{X}_2$ on  \[
N\Sigma \equiv \{(x_1,x_2,t)\in \R^3\colon\, \log(x_1^2+x_2^2)\in I\}, 
\]
which we view as a domain in \begin{equation}\label{eq:E} E=\R^3\setminus \{x_1=x_2=0\}.
\end{equation}
More precisely, given $f\in L^2(N\Sigma, \lambda)$ ($\lambda:=$Lebesgue measure) we want $v\in L^2(N\Sigma, \lambda)$ such that:\begin{itemize}
	\item[(DP1)] $\mathfrak{X}_jv\in L^2(N\Sigma, \lambda)$ for $j=1,2$, 
	\item[(DP2)] $\mathfrak{P}v=f$ in the interior of $N\Sigma$,
	\item[(DP3)] $v=0$ on the boundary $\partial N\Sigma= \{(x_1,x_2,t)\in \R^3\colon\, \log(x_1^2+x_2^2)\in \partial I\}$. \end{itemize}
The latter condition is interpreted as the requirement that, extending $v$ by zero off $N\Sigma$, $\mathfrak{Y}v\in L^2(E, \lambda)$. Here $\mathfrak{Y}$ is the vector field \eqref{eq:Y_model} tranversal to the boundary circles of the annulus. We call any $v$ as above a solution of the \emph{model Dirichlet problem} with datum $f$.

By construction, $\mathfrak{X}_1$ and $\mathfrak{X}_2$ are tangent to $\Sigma\equiv \Sigma\times \{0\}$. The latter surface  cuts the domain into an upper half 
\[N_+\Sigma = \{(x_1,x_2,t)\in E\colon\, \log(x_1^2+x_2^2)\in I, \, t>0\}\] and a lower half 
\[N_-\Sigma = \{(x_1,x_2,t)\in E\colon\, \log(x_1^2+x_2^2)\in I, \, t<0\}.\] 
Given $f\in L^2(N_+\Sigma, \lambda)$, we say that $v$ is a solution of the \emph{upper model Dirichlet problem} with datum $f$ if conditions (DP1) and (DP2) hold with $N_+\Sigma$ in place of $N\Sigma$ and $v=0$ on the upper part of the boundary, namely \[\partial_+ N\Sigma= \{(x_1,x_2,t)\in \R^3\colon\, \log(x_1^2+x_2^2)\in \partial I, \quad t>0\}.\] Of course, this is interpreted as the square integrability of $\mathfrak{Y}v$ on $E\cap \{t>0\}$, where $v$ is extended by zero off $N_+\Sigma$. Similarly, given $f\in L^2(N_-\Sigma, \lambda)$ we have a notion of solution of the \emph{lower model Dirichlet problem} with datum $f$, obtained by changing appropriately "$+$" into "$-$" and "$t>0$" into "$t<0$" in the above lines. 

We have the following gluing lemma, expressing the fact that in the model Dirichlet problem no information propagates from one side to the other of $\Sigma$.

\begin{lem}[Gluing lemma]\label{lem:gluing}
Given $f\in L^2(N\Sigma, \lambda)$, define $f_+$ as the restriction of $f$ to $N_+\Sigma$ and similarly $f_-$ as the restriction to $N_-\Sigma$. If $v_+$ solves the upper model Dirichlet problem with datum $f_+$ and $v_-$ solves the lower model Dirichlet problem with datum $f_-$, then \[
v:=\begin{cases}
	v_+\qquad &\text{on }N_+\Sigma\\
	v_-\qquad &\text{on }N_-\Sigma
	\end{cases}
\] solves the model Dirichlet problem with datum $f$. 
	\end{lem}

\begin{proof}
By assumption, $\mathfrak{X}_jv_+\in L^2(N_+\Sigma, \lambda)$ and $\mathfrak{X}_jv_-\in L^2(N_-\Sigma,\lambda)$ for $j=1,2$. With $v$ as in the statement, this means that there are $v_j\in L^2(N\Sigma, \lambda)$ such that \begin{equation}\label{eq:gluing_1}
\int_{N\Sigma} v\mathfrak{X}_j^\dagger g\, d\lambda = \int_{N\Sigma} v_j g\, d\lambda
\end{equation} for every test function $g$ supported in the interior of $N\Sigma$ and \emph{vanishing near $\Sigma$}. Let $g$ be an arbitrary test function on $N\overset{\circ}{\Sigma}$. Choose a nonnegative $\chi\in C^\infty(\R)$ such that $\chi\equiv 0$ near $0$ and $\chi(t)\equiv 1$ for $|t|\geq 1$. Put $\chi_\epsilon(x_1,x_2, t)=\chi(\epsilon^{-1}t)$. Applying \eqref{eq:gluing_1} to the test function $\chi_\epsilon g$, we obtain \[
\int_{N\Sigma} v\chi_\epsilon\mathfrak{X}_j^\dagger g\, d\lambda-\int_{N\Sigma} vg\mathfrak{X}_j\chi_\epsilon\, d\lambda = \int_{N\Sigma} v_j \chi_\epsilon g\, d\lambda.
\]
Since $\mathfrak{X}_j$ is tangent to $\Sigma$, the functions $\mathfrak{X}_j\chi_\epsilon$ stay uniformly bounded as $\epsilon$ tends to zero. Thus, an application of the dominated convergence theorem shows that the limit for $\epsilon$ tending to zero of the above identity is \[
\int_{N\Sigma} v\mathfrak{X}_j^\dagger g\, d\lambda= \int_{N\Sigma} v_j  g\, d\lambda.
\]
This proves condition (DP1) in the definition of solution of the model Dirichlet problem. The facts that $v_+$ solves the model equation in the upper domain with datum $f_+$ and $v_-$ solves the model equation in the lower domain with datum $f_-$ can now be restated as the identity \[
\sum_{j=1}^2 \int_{N\Sigma} v_j  \mathfrak{X}_j g\, d\lambda =  \int_{N\Sigma} f  g\, d\lambda
\] for every test function $g$ supported in the interior of $N\Sigma$ and vanishing near $\Sigma$. Now we may argue exactly as above to deduce the same identity for every test function $g$ supported in the interior of $N\Sigma$. This completes the proof of condition (DP2). Condition (DP3) is proved similarly. 
\end{proof}

By the invariance of the vector fields $\mathfrak{X}_1$ and $\mathfrak{X}_2$ under the symmetry $t\mapsto -t$, it is enough to understand the upper model Dirichlet problem. In fact, it is straightforward to verify that $v(x_1,x_2, t)$ solves the lower model Dirichlet problem with datum $f(x_1,x_2,t)$ if and only if $v(x_1,x_2, -t)$ solves the upper model Dirichlet problem with datum $f(x_1,x_2,-t)$. In view of this and Lemma \ref{lem:gluing}, in the following two sections we restrict our attention to the upper model Dirichlet problem. 

\subsection{Foliation and adapted coordinates} In this section, we find a new system of coordinates better adapted to the geometry of the problem. The key observation is that the two vector fields \eqref{eq:X_model} commute: \begin{equation}\label{eq:commutation}
[\mathfrak{X}_1, \mathfrak{X}_2]=0. 
\end{equation}
A conceptual explanation of this fact is postponed to section \ref{sec:geometry}. Identity \eqref{eq:commutation} holds on the domain of definition of the two vector fields, namely $E$ (cf.\eqref{eq:E}). By Frobenius' theorem, there is a foliation of $E$ whose leaves have tangent spaces spanned by $\mathfrak{X}_1, \mathfrak{X}_2$. Thus, $\mathfrak{P}$ is what is called a foliation Laplacian. The surface \[
E\cap \{t=0\}
\] is clearly a leaf (containing the annulus $\Sigma$). This leaf cuts $E$ into the two halves \[
E_+=E\cap \{t>0\}, \quad E_-=E\cap \{t<0\}, 
\]
on which the original foliation induces two separate foliations. By the already remarked invariance under the symmetry $t\mapsto -t$, it is enough to examine $E_+$.

We will introduce new coordinates on $E_+$, which exhibit the simple structure of the foliation. We will be guided by the formula \eqref{eq:model_equation_formula}. One may notice that the vector field \[
\mathfrak{T}=-\frac{x_2}{x_1^2+x_2^2}\frac{\partial}{\partial x_1}+ \frac{x_1}{x_1^2+x_2^2}\frac{\partial}{\partial x_2}
\]
differentiates in the angular direction and the transversal vector field $\mathfrak{Y}$ differentiates in the radial direction. This suggests using polar coordinates $(r,\theta)$ in the $(x_1,x_2)$-plane. A simple computation shows that \[\mathfrak{T}=r^{-2}\frac{\partial}{\partial \theta}, \quad \mathfrak{Y}=r\frac{\partial}{\partial r}.\] By the familiar formula 
\begin{eqnarray*}
	\Delta_{x_1,x_2}&=&\frac{\partial^2}{\partial r^2}+r^{-1}\frac{\partial}{\partial r}+r^{-2}\frac{\partial^2}{\partial \theta^2}\\
	&=&r^{-2}\left[\left(r\frac{\partial}{\partial r}\right)^2+\frac{\partial^2}{\partial \theta^2}\right], 
\end{eqnarray*} 
we get  \[
\mathfrak{P}=r^{-2}\left\{-\frac{1}{4}\left(r\frac{\partial}{\partial r}\right)^2-\frac{1}{4}\frac{\partial^2}{\partial \theta^2}+\frac{\alpha}{2}\frac{\partial}{\partial \theta}\left(2t\frac{\partial}{\partial t}+1\right)-\alpha^2\left(t\frac{\partial}{\partial t}+1\right)t\frac{\partial}{\partial t}\right\}.
\]
In view of the dilation invariant structure in $t$ and $r$ variables, we now write \[
r=e^\sigma, \quad t=e^\nu \qquad \text{with } (\sigma, \nu)\in \R^2,
\] so that $\mathfrak{Y}=r\frac{\partial}{\partial r}=\frac{\partial}{\partial \sigma}$ and $t\frac{\partial}{\partial t}=\frac{\partial}{\partial \nu}$. Notice that $(\sigma, \nu, \theta)$ are global coordinates giving the identification \begin{equation}\label{eq:N_+}
E_+ \simeq \R\times \R\times \mathbb{S}^1.
\end{equation}
In this identification, the leaves are the planes $\R\times \R\times \{\theta\}$ ($\theta\in \mathbb{S}^1$) and $\mathbb{S}^1$ is the leaf space. In the $(\sigma, \nu, \theta)$-coordinates \[
\mathfrak{P}=e^{-2\sigma}\left\{-\frac{1}{4}\frac{\partial^2}{\partial \sigma^2}-\frac{1}{4}\frac{\partial^2}{\partial \theta^2}+\frac{\alpha}{2}\frac{\partial}{\partial \theta}\left(2\frac{\partial}{\partial \nu}+1\right)-\alpha^2\left(\frac{\partial}{\partial \nu}+1\right)\frac{\partial}{\partial \nu}\right\}.\]
Completing a square we may rewrite this as 
	\[\mathfrak{P}=-\frac{e^{-2\sigma}}{4}\left\{\frac{\partial^2}{\partial \sigma^2}+\left(2\alpha\frac{\partial}{\partial \nu}-\frac{\partial}{\partial \theta}\right)^2+2\alpha\left(2\alpha\frac{\partial}{\partial \nu}-\frac{\partial}{\partial \theta}\right)\right\}.
\]
This operator involves only differentiation in the $\sigma$-direction, namely the first factor in the decomposition \eqref{eq:N_+}, and differentiation along the vector field \[\mathfrak{D}=2\alpha\frac{\partial}{\partial \nu}-\frac{\partial}{\partial \theta},\] which is tangent to the product of the two other factors, the cylinder $\R\times \mathbb{S}^1$. A further change of variables aligns the vector field $\mathfrak{D}$ with the axis of the cylinder. Let then \[
\tau=\frac{\nu}{2\alpha}\in \R,\qquad  \phi=\theta+\frac{\nu}{2\alpha}\in \mathbb{S}^1, 
\]
so that \[
\mathfrak{D}=\frac{\partial}{\partial \tau}. 
\]
We may also notice that \begin{eqnarray*}
	\mathrm{span}\{\mathfrak{X}_1, \mathfrak{X}_2\}&=& 	\mathrm{span}\{2x_1\mathfrak{X}_1+2x_2\mathfrak{X}_2,2x_2\mathfrak{X}_1-2x_1\mathfrak{X}_2\}\\
	&=&\mathrm{span}\left\{\mathfrak{Y}, \mathfrak{D}\right\}=\mathrm{span}\left\{\frac{\partial}{\partial \sigma},  \frac{\partial}{\partial \tau}\right\}.
	\end{eqnarray*}
We record the results of our computations in a proposition, where we also include the expression of Lebesgue measure $\lambda$ in the new coordinates, for later reference. The proof of the latter is a straightforward computation, which we omit.  

\begin{prp}[Formula for $\mathfrak{P}$]\label{prp:model_worm}
	On $E_+$ there are global coordinates \[
	\sigma=\log \sqrt{x_1^2+x_2^2}, \quad \tau=\frac{\log t}{2\alpha}, \quad \phi=\theta+\frac{\log t}{2\alpha}, 
	\]
	where $\theta\in \mathbb{S}^1$ is the angular coordinate in the $(x_1,x_2)$-plane. These coordinates give a diffeomorphism $E_+\simeq \R^2\times \mathbb{S}^1$ and the representations \[
	\mathfrak{P}=-\frac{e^{-2\sigma}}{4}\left(\frac{\partial^2}{\partial \sigma^2}+\frac{\partial^2}{\partial \tau^2}+2\alpha\frac{\partial}{\partial \tau}\right), \qquad \mathfrak{Y}=\frac{\partial}{\partial \sigma}.\]
	Moreover, \[
		\mathrm{span}\{\mathfrak{X}_1, \mathfrak{X}_2\}=\mathrm{span}\left\{\frac{\partial}{\partial \sigma},  \frac{\partial}{\partial \tau}\right\},
	\] and the measure $\lambda$ has density $2|\alpha|e^{2\sigma+2\alpha\tau}$ in the new coordinates. 
	\end{prp}

\subsection{Fourier analysis of the upper model Dirichlet problem}\label{sec:fourier}

We now investigate solutions of the upper model Dirichlet problem \begin{equation}\label{eq:model_equation2}
\mathfrak{P}v=f.
\end{equation}
Recall that $f,v\in L^2(N_+\Sigma, \lambda)$, $\mathfrak{X}_jv\in L^2(N_+\Sigma, \lambda)$ for $j=1,2$, and we impose the boundary condition \[
v=0 \quad \text{on}\quad  \partial N_+\Sigma, 
\]
interpreted extending $v$ by zero on $E_+\setminus N_+\Sigma$ and requiring that $\mathfrak{Y}v\in L^2(E_+, \lambda)$. 

Write $I=[2a, 2b]$, $a<b$. By Proposition \ref{prp:model_worm}, \eqref{eq:model_equation2} translates into the equation \begin{equation}\label{eq:model_equation_strip}
-\left(\frac{\partial^2}{\partial \sigma^2}+\frac{\partial^2}{\partial \tau^2}+2\alpha\frac{\partial}{\partial \tau}\right)v=4e^{2\sigma}f=:f_1
\end{equation}
on the domain \[[a,b]_\sigma\times \R_{\tau}\times \mathbb{S}^1_\phi.\] Notice that the operator, the density of the measure $\lambda$, and all the conditions on the datum and the solution, are independent of the angular variable $\phi$, so our problem effectively reduces to the Dirichlet problem for the operator \eqref{eq:model_equation_strip} on the strip \[
S=[a,b]_\sigma\times \R_{\tau}.\] Accordingly, we fix the angle $\phi$ and write $f_1(\sigma, \tau)=f_1(\sigma, \tau, \phi)$ and $v(\sigma, \tau)=v(\sigma, \tau, \phi)$. In $(\sigma, \tau)$-coordinates, the conditions on datum and solution read \begin{eqnarray*}
&\int_S |e^{\alpha\tau}f_1|^2\,  d\sigma\, d\tau, \quad \int_S |e^{\alpha\tau}v|^2\, d\sigma\, d\tau <+\infty\\
& \int_S \left|e^{\alpha\tau}\frac{\partial v}{\partial \tau}\right|^2\, d\sigma\, d\tau, \quad \int_{\R^2} \left|e^{\alpha\tau}\frac{\partial v}{\partial \sigma}\right|^2\, d\sigma\, d\tau<+\infty,
\end{eqnarray*}
where we dropped the innocuous factors $2|\alpha| e^{2\sigma}$ and $v$ is extended by zero off the strip. Write 
\[
V(\sigma, \tau)=e^{\alpha\tau}v(\sigma, \tau), \quad F(\sigma, \tau)=e^{\alpha\tau}f_1(\sigma, \tau).
\]
Since 
\[
\left(\frac{\partial^2}{\partial \sigma^2}+\frac{\partial^2}{\partial\tau^2}-\alpha^2\right)V=e^{\alpha\tau}\left(\frac{\partial^2}{\partial \sigma^2}+\frac{\partial^2}{\partial\tau^2}+2\alpha\frac{\partial}{\partial\tau}\right)v, 
\]
\eqref{eq:model_equation_strip} becomes 
\begin{equation}\label{eq:schrodinger}
\mathfrak{Q}V=F, \qquad \mathfrak{Q}:=-\frac{\partial^2}{\partial \sigma^2}-\frac{\partial^2}{\partial \tau^2}+\alpha^2
\end{equation} subject to the conditions \begin{equation}\label{eq:schroding_conditions}
F, V,\frac{\partial V}{\partial \tau}\in L^2(S), \qquad \frac{\partial V}{\partial \sigma}\in L^2(\R^2), 
\end{equation}
where the measure on the strip is Lebesgue measure. This problem can now be solved by a standard Fourier decomposition. Let
\[
\psi_n(\sigma) = \sqrt{\frac{2}{b-a}}\sin\left(n\pi \frac{\sigma-a}{b-a}\right)\qquad (n=1,2, \ldots, )
\]
be the standard orthonormal basis of Dirichlet eigenfunctions of $-\frac{\partial^2}{\partial \sigma^2}$ on $L^2([a,b])$, corresponding to eigenvalues \[E_n=\frac{\pi^2}{(b-a)^2}n^2.\] 
In particular, if $\psi_n$ is extended by zero on $\R\setminus [a,b]$ then \begin{equation}\label{eq:psi_n_dirichlet}\left\lVert \frac{\partial \psi_n}{\partial \sigma}\right\rVert_{L^2(\R)}^2=E_n.\end{equation} 
Fourier transforming in $\tau$ and expanding in the above basis in $\sigma$, we may represent \[
V(\sigma, \tau)=\frac{1}{2\pi}\sum_{n\geq 1}\psi_n(\sigma)  \int_\R \widehat{V}(n,\omega) e^{i \omega \tau}\, d\omega, 
\]
where \[
\widehat{V}(n,\omega)=\int_S V(\sigma, \tau) \psi_n(\sigma) e^{-i\omega\tau}\, d\sigma\, d\tau, 
\]
and similarly for $F$. In view of the Fourier identity
\[
\mathfrak{Q}V = \frac{1}{2\pi}\sum_{n\geq 1}\psi_n(\sigma)  \int_\R \left(E_n+\omega^2+\alpha^2\right)\widehat{V}(n,\omega) e^{i \omega \tau}\, d\omega, 
\]
equation \eqref{eq:schrodinger} is formally solved by setting \[
\widehat{V}(n,\omega):=\frac{\widehat{F}(n,\omega)}{E_n+\omega^2+\alpha^2}\qquad \forall n\geq 1, \, \omega\in \R. 
\]
That is, we have the formula \begin{equation}\label{eq:fourier_inversion}
V(\sigma, \tau)=\frac{1}{2\pi}\sum_{n\geq 1}\psi_n(\sigma)  \int_\R \frac{\widehat{F}(n,\omega)}{E_n+\omega^2+\alpha^2} e^{i \omega \tau}\, d\omega. 
\end{equation}
Now Plancherel's theorem and \eqref{eq:psi_n_dirichlet} allow us to rephrase the integrability and boundary conditions on datum and solution as follows:
 \[
\sum_{n\geq 1}\int_\R (E_n+\omega^2)|\widehat{V}(n,\omega)|^2\, d\omega<+\infty, \qquad \sum_{n\geq 1}\int_\R |\widehat{F}(n,\omega)|^2\, d\omega<+\infty. 
\]
From this, one easily checks that formula \eqref{eq:fourier_inversion} defines the unique solution of \eqref{eq:schrodinger} subject to the boundary conditions \eqref{eq:schroding_conditions}. The solution to the original problem \eqref{eq:model_equation2} is therefore unique, and can be obtained by applying formula \eqref{eq:fourier_inversion} to  \[
F_\phi(\sigma, \tau)=4e^{2\sigma + \alpha\tau}f(\sigma, \tau, \phi), 
\] denoting the result $V_\phi(\sigma, \tau)$, and setting \[
v(\sigma, \tau, \phi)=e^{-\alpha\tau}V_\phi(\sigma, \tau). 
\]
By Section \ref{sec:gluing}, the solution of the model Dirichlet problem on the whole normal bundle $N\Sigma$ is also unique, and is found by solving two upper model Dirichlet problems as we just did. 

\subsection{Asymptotic expansion near the annulus}

We now use the Fourier analysis of the last section to prove the following key proposition. 

\begin{prp}\label{prp:asymptotic} Let $f\in C^2_c(N_+\Sigma)$. Let $v$ be the solution of the model Dirichlet problem with datum $f$. Then:\begin{enumerate}
		\item $v\equiv 0$ on $N_-\Sigma$; 
		\item there are coefficients $c_n(f; \phi)$, continuous and $2\pi$-periodic in $\phi$, such that on $N_+\Sigma$ we have the asymptotic expansion \[
	v \sim \sum_{n\geq 1}c_n\left(f; \theta+\frac{\log t}{2\alpha}\right)\sin\left(n\pi \frac{\log r-a}{b-a}\right)  t^{\frac{1}{2}(\sqrt{1+E_n/\alpha^2}-1)}, 
	\]
	valid for small positive $t$ in the sense, after truncation at the $n$-th term, the error is $O(t^{\frac{1}{2}(\sqrt{1+E_{n+1}/\alpha^2}-1)})=o(t^{\frac{1}{2}(\sqrt{1+E_n/\alpha^2}-1)})$. 
	\end{enumerate} Moreover, if $f\geq 0$ is not identically zero, then $c_1(f; \phi)\geq 0$ is not identically zero in $\phi$. 
	\end{prp}

\begin{proof} 
	By Lemma \ref{lem:gluing}, it is enough to solve the upper and lower model Dirichlet problems separately. Since the datum vanishes on $N_-\Sigma$, the solution is zero there and it is enough to solve the upper model Dirichlet problem. We work in the coordinates of Proposition \ref{prp:model_worm}, proceeding as in Section \ref{sec:fourier}. We have the family of functions \[
	F_\phi(\sigma, \tau)=4e^{2\sigma + \alpha\tau}f(\sigma, \tau, \phi), 
	\]
	parametrized by the angle $\phi$. Each $F_\phi$ is a $C^2$ function compactly supported on the strip $S=[a,b]\times \R$. In particular, the Fourier coefficient \[
	\widehat{F_\phi}(n,\omega)=\int_S F_\phi(\sigma, \tau) \psi_n(\sigma) e^{-i\omega\tau}\, d\sigma\, d\tau
	\] is entire in $\omega$ for each $n$. Write $\omega=x+iy$ and let $C>0$ be such that $f(\sigma, \tau, \phi)$ vanishes for $|\tau|\geq C$. Integrating by parts in $\sigma$ we obtain the useful bound \begin{eqnarray}
\notag		\widehat{F_\phi}(n,\omega)&=&\int_S F_\phi(\sigma, \tau) \psi_n(\sigma) e^{(y-ix)\tau}\, d\sigma\, d\tau\\
\notag		&=&E_n^{-1}\int_S \left(-\frac{\partial^2}{\partial \sigma^2}\right)F_\phi(\sigma, \tau) \psi_n(\sigma) e^{(y-ix)\tau}\, d\sigma\, d\tau\\
\label{eq:fourier_coefficient_bound}		&=&O(E_n^{-1}e^{C|y|}).
	\end{eqnarray}
	As seen in the previous section, the solution of the upper model Dirichlet problem is \[
	v(\sigma, \tau, \phi)=\frac{e^{-\alpha\tau}}{2\pi}\sum_{n\geq 1}\psi_n(\sigma)a_n(\tau, \phi), 
	\]
	where 
\begin{equation}\label{eq:complex_integral}
a_n(\tau, \phi)=\int_\R \frac{\widehat{F_\phi}(n,\omega)}{E_n+\omega^2+\alpha^2} e^{i \omega \tau}\, d\omega.
\end{equation}
The integrand is a meromorphic function with poles at $\omega=\pm i\sqrt{E_n+\alpha^2}$. Remember that we are interested in the small-$t$ asymptotics. Since $\tau=\frac{\log t}{2\alpha}$, this corresponds to the regime $\tau\rightarrow-\infty$ when $\alpha>0$ and to the regime $\tau\rightarrow +\infty$ when $\alpha<0$. Accordingly, we move the contour of integration in \eqref{eq:complex_integral} to the line $\R-i\mathrm{sgn}(\alpha)T$, where $T>0$. This is easily justified by \eqref{eq:fourier_coefficient_bound}. If $T>\sqrt{E_n+\alpha^2}$, we cross a pole on the imaginary axis and obtain  \[
a_n(\tau, \phi) = -\mathrm{sgn}(\alpha)2\pi ie^{\mathrm{sgn}(\alpha)\tau\sqrt{E_n+\alpha^2}}\frac{\widehat{F_\phi}(n,-i\mathrm{sgn}(\alpha)\sqrt{E_n+\alpha^2})}{-2i\mathrm{sgn}(\alpha)\sqrt{E_n+\alpha^2}}+ \int_{\R-i\mathrm{sgn}(\alpha)T} \frac{\widehat{F_\phi}(n,\omega)}{E_n+\omega^2+\alpha^2} e^{i \omega \tau}\, d\omega, 
\]
while for $0<T<\sqrt{E_n+\alpha^2}$, we have \[
a_n(\tau, \phi) =  \int_{\R-i\mathrm{sgn}(\alpha)T} \frac{\widehat{F_\phi}(n,\omega)}{E_n+\omega^2+\alpha^2} e^{i \omega \tau}\, d\omega.\]
Again by \eqref{eq:fourier_coefficient_bound}, \begin{eqnarray*}
	\int_{\R-i\mathrm{sgn}(\alpha)T} \frac{\widehat{F_\phi}(n,\omega)}{E_n+\omega^2+\alpha^2} e^{i \omega \tau}\, d\omega &=& \int_\R \frac{\widehat{F}(n,x-i\mathrm{sgn}(\alpha)T)}{E_n+\alpha^2-T^2+x^2-2ixT} e^{(\mathrm{sgn}(\alpha)T+ix) \tau}\, dx\\
	&=& O\left(E_n^{-1}e^{CT+\mathrm{sgn}(\alpha)T\tau}\int_\R (1+x^2)^{-1}\, dx\right)\\
&=&O(E_n^{-1}e^{CT+\mathrm{sgn}(\alpha)T\tau}).
	\end{eqnarray*}
Let \[
\sqrt{E_{n_0}+\alpha^2}<T<\sqrt{E_{n_0+1}+\alpha^2}.
\]
Then \begin{eqnarray*}
	v(\sigma, \tau, \phi)&=&\frac{e^{-\alpha\tau}}{2}\sum_{n=1}^{n_0}\psi_n(\sigma)  e^{\mathrm{sgn}(\alpha)\tau\sqrt{E_n+\alpha^2}}\frac{\widehat{F_\phi}(n,-i\mathrm{sgn}(\alpha)\sqrt{E_n+\alpha^2})}{\sqrt{E_n+\alpha^2}}\\
&+&e^{-\alpha\tau}\sum_{n\geq 1}\psi_n(\sigma)O(E_n^{-1}e^{CT+\mathrm{sgn}(\alpha)T\tau})\\
&=&\frac{1}{2}\sum_{n=1}^{n_0}\psi_n(\sigma)  e^{\mathrm{sgn}(\alpha)(\sqrt{E_n+\alpha^2}-|\alpha|)\tau}\frac{\widehat{F_\phi}(n,-i\mathrm{sgn}(\alpha)\sqrt{E_n+\alpha^2})}{\sqrt{E_n+\alpha^2}}+O_T(e^{\mathrm{sgn}(\alpha)(T-|\alpha|)\tau}), 
	\end{eqnarray*}
where we used that $\sum_{n\geq 1}E_n^{-1}<\infty$. Let \begin{eqnarray*}
c_n(f; \phi)&:=&\frac{1}{\sqrt{2b-2a}}\frac{\widehat{F_\phi}(n,-i\mathrm{sgn}(\alpha)\sqrt{E_n+\alpha^2})}{\sqrt{E_n+\alpha^2}}\\
&=&\frac{1}{\sqrt{2b-2a}}\frac{1}{\sqrt{E_n+\alpha^2}}\int_S e^{2\sigma + \alpha\tau}f(\sigma, \tau, \phi) \psi_n(\sigma) e^{-\mathrm{sgn}(\alpha)\sqrt{E_n+\alpha^2}\tau}\, d\sigma\, d\tau
\end{eqnarray*}
Recalling that $t=e^{2\alpha\tau}$ and $\sigma=\log r$, we obtain the desired asymptotic expansion. The observation about the first coefficient follows from the positivity of $\psi_1(\sigma)$. 

\end{proof}

\subsection{Proof of failure of exact regularity} We can finally disprove exact regularity of the worm sublaplacians $P$ introduced in Section \ref{sec:worm_vector_measure}. Refer to that section for the notation. Recall that for every datum $f\in L^2(M,\mu)$ of zero average there exists a unique $u\in L^2(M,\mu)$ of zero average and finite energy (i.e., $X_1u, X_2u\in L^2(M, \mu)$) such that $Pu=f$ (by Theorem \ref{thm:global_solutions_cr}, see also the end of Section \ref{sec:worm_vector_measure}). Write $u:=P^{-1}f$. 

\begin{thm}[Worm sublaplacians are not exactly regular]\label{thm:worm_irreg}
For any worm sublaplacian $P$, there exists $k_0\in \mathbb{N}$ such that $P$ is not exactly regular of order $k_0$. 
	\end{thm}

\begin{proof} 
Assume, by contradiction, that $P$ is exactly regular of all orders, that is, for every $k\in \mathbb{N}$ and all data $f\in H^k(M)$ of zero average, $P^{-1}f\in H^k(M)$. As remarked at the end of Section \ref{sec:worm_vector_measure}, the worm sublaplacian satisfies all the assumptions of the scaling lemma (Lemma \ref{lem:scaling}), where of course $\Sigma$ is the embedded complex-analytic annulus. Thus, for each $f\in H^k_\mathrm{comp}(N\Sigma)$ the model Dirichlet problem of Section \ref{sec:gluing} admits solutions satisfying the vertical regularity condition $v\in H^k_\mathrm{vert}(N\Sigma)$. We also saw in Section \ref{sec:fourier} that such solutions are unique. 

Let $f$ be a nontrivial nonnegative test function compactly supported in $N_+\Sigma$ and $v$ the solution of the model Dirichlet problem with datum $f$. By what we said, $v\in H^k_\mathrm{vert}(N\Sigma)$ for all $k$. By the Sobolev embedding, $v(x_1,x_2, \cdot)$ is smooth for almost every $x_1,x_2$. By Proposition \ref{prp:asymptotic}, $v$ vanishes on $N_-\Sigma$, while on $N_+\Sigma$ we have \begin{equation}\label{eq:asymptotic}
v = c_1\left(f; \theta+\frac{\log t}{2\alpha}\right)\sin\left(\pi \frac{\log r-a}{b-a}\right)  t^{\frac{1}{2}(\sqrt{1+E_1/\alpha^2}-1)}+o(t^{\frac{1}{2}(\sqrt{1+E_1/\alpha^2}-1)}), 
\end{equation}
where $E_1=\frac{\pi^2}{(b-a)^2}$ and $c_1(f; \phi)>0$ for all $\phi$ on a interval of positive length (by continuity). Since $v$ vanishes for $t<0$, vertical smoothness forces \begin{equation}\label{eq:flatness}
	v(x_1,x_2, t)=O(t^N)\qquad \forall N\in \mathbb{N}\end{equation} 
for almost all $(x_1,x_2)$. Evaluating \eqref{eq:asymptotic} at $(x_1,x_2, t)$ for $t=e^{-4\pi |\alpha| k}$, $k\in \mathbb{N}$, and $(x_1,x_2)$ of polar coordinates $(r,\phi_0)$, we get \[
v = c_1\left(f; \phi_0\right)\sin\left(\pi\frac{\log r-a}{b-a}\right)t^{\frac{1}{2}(\sqrt{1+E_1/\alpha^2}-1)}+o(t^{\frac{1}{2}(\sqrt{1+E_1/\alpha^2}-1)})
\] 
for arbitrarily small $t>0$. Choosing, as we may, $r$ and $\phi$ so that \eqref{eq:flatness} holds for the corresponding $(x_1,x_2)$ gives the desired contradiction. 
	\end{proof}

\subsection{A remark on the foliation on $N\Sigma$}\label{sec:geometry}

We sketch a differential geometric explanation for the appearance of a foliation in the model problem. Let $\xi$ be a nowhere vanishing one-form such that $HM=\ker \xi$, defined modulo a conformal factor $\pm e^f$, with $f\in C^\infty(M, \R)$. Then pseudoconvexity of $M$ amounts to the condition \begin{equation}\label{eq:confoliation}
	\xi \wedge d\xi\geq 0, \end{equation} which may be easily seen to be independent of the conformal factor (here $M$ inherits its orientation as the boundary of a domain in $\mathbb{C}^2$). At points of strong pseudoconvexity the inequality is strict, that is, $\xi$ is a contact form. A plane field on a three-dimensional oriented manifold satisfying \eqref{eq:confoliation} (for some, hence every, nowhere zero one-form $\xi$ annihilating the plane field) is called a \emph{confoliation}, after \cite{EliashbergThurston1998}. By a formulation of the three-dimensional Frobenius theorem, if equality holds in \eqref{eq:confoliation} on an open set, then the plane field gives a foliation there. Thus, confoliations interpolate between foliations and contact structures. Assume that a confoliation has a leaf, that is, an integral surface $\Sigma$ tangent to the confoliation. If $\Phi$ is a tubular neighborhood diffeomorphism (as in Section \ref{sec:dilation_machinery}), we may write \[
(\Phi^{-1})^*\xi=dt+t\pi^*\omega+t^2\omega', 
\]
where $\omega$ is a one-form on $\Sigma$, $\pi:N\Sigma\rightarrow \Sigma$ is the canonical projection, and $\omega'$ is a one-form on $N\Sigma$. Then \begin{eqnarray*}
(\Phi^{-1})^*\xi \wedge d((\Phi^{-1})^*\xi)&=&(dt+t\pi^*\omega+t^2\omega')\wedge (dt\wedge\pi^*\omega+t\pi^*d\omega+2tdt\wedge \omega'+t^2d\omega')\\
&=&tdt\wedge \pi^*d\omega+t^2\omega'', 
\end{eqnarray*}
where $\omega''$ is a one-form on $N\Sigma$. By the confoliation condition, this three-form is nonnegative in a neighborhood of the zero section, which forces $dt\wedge \pi^*d\omega$ to vanish. Since $\omega$ is a one-form on $\Sigma$, this means that $\omega$ is closed. When $M$ is the boundary of smooth domain in $\mathbb{C}^2$ and $\Sigma$ is an embedded Riemann surface with boundary, then $[\omega]\in H^1_\mathrm{dR}(\Sigma, \R)$ is the D'Angelo class. In particular, the fact that $d\omega=0$ is a reformulation of an observation of Boas and Straube \cite{BoasStraube1993}. 

The conclusion is that the $1$-jet of the confoliation $\xi$ in the fiber variable $t$, namely $dt+t\pi^*\omega$, is a foliation. A computation shows that, in the case of the worm domain, the closed form $\omega$ is a multiple of the angle form on the annulus. In general, the leaves of the foliation are easily seen to be the zero section $\{t=0\}$ and the multivalued graphs of the functions $\pm e^g$, as $g$ varies over the primitives of the closed form $\omega$. These foliations have linear holonomy around the zero section, and the holonomy is nontrivial if and only if $\omega$ is not globally exact. 

It would be of great interest to carry out the analysis of Section \ref{sec:worm} for the generalized worms of \cite{DallAraMongodi2023}, which leads to the more general foliations just discussed.



\end{document}